\documentclass[11pt]{article}

\usepackage[english]{babel} 
    \usepackage[center]{caption}

\usepackage{subcaption}

\usepackage[style = numeric, backend=biber, mincitenames=1, maxcitenames=5, maxbibnames = 5, refsection=section, hyperref=true, backref=true]{biblatex}

\usepackage[titletoc]{appendix}

\usepackage{hyperref}

\usepackage[thms]{packages}

\title{The OSSS Method in Percolation Theory}
\author{Julian Kern}

\begin{document}

\maketitle

\begin{abstract}
In 2017, Duminil-Copin \emph{et al.}~introduced the OSSS method to study properties of diverse percolation models. This document aims to introduce the reader to this new method. It contains a introduction to percolation theory, then concentrates on the case of Poisson-Boolean percolation. The majority of this document is dedicated to an detailed analysis of \cite{DRT18}. This work is the result of an internship in Summer 2019 with Jean-Baptiste Gouéré at the University of Tours.
\end{abstract}

\tableofcontents

\newpage

\section{Introduction}

Percolation Theory is a one of those branches of mathematics in which very easily stated questions turn out to be extremely difficult to answer. The first percolation model was introduced 1959 in \cite{BH59} by Broadbent and Hammersley to understand the dispersion of water in porous material. The main question can be formulated as follows: "Given a porous stone, will the water penetrate to its centre or will it enter the stone only superficially?" To turn this problem into a mathematical model, let us simplify it. On the microscopic level we may assume that we can decide at each point if water can flow or if it cannot. On this microscopic scale, the stone is nearly infinitely large and we may model it via the lattice on $\mathbb{Z}^3$. We then define the edges in the lattice to be \emph{open} or \emph{closed} when water can or cannot flow at the given point respectively. A first attempt will be to declare the edges to be open with probabilities $p$ and closed with probability $1-p$ (identically and independently). In this context where the stone is taken to be infinitely large in proportion to the size of one edge, the above question reduces to: "Is there an open path which connects the origin to infinity?" It is easy to show that there is a probability threshold $p_c$ which separates the two phases of \emph{non percolation} and \emph{percolation}, i.e.~a phase in which the origin is not connected to infinity almost surely and a phase in which it is connected to infinity with positive probability. However, it turns out to be very difficult to understand the \emph{transition} between these two phases. The study of the different phases and this \emph{phase transition} which will be the main objective of Percolation Theory. The first major advances have been made in the '80s. Today this first \emph{Bernoulli} model is understood quite well, but even now some questions remain unanswered and are subject to current research. During the last sixty years, new models have been developed, amongst which are the discrete \emph{Gibb's model} and the continuous \emph{Poisson-Boolean model} which both add long range dependencies to the model.

During my internship, I studied a new method which allows to deduce (relatively) easily many properties of the phase transition at once. This so-called \emph{OSSS method} has been introduced in \cite{DRT17a} by Duminil-Copin, Raoufi and Tassion to study different discrete models including Gibb's model. It heavily relies on the OSSS inequality first proved in \cite{OSSS05} and appeared to be robust with respect to changes in the model. In the last two years this has been confirmed, as the OSSS method has been applied to a large class of percolation models including Voronoi, Poisson-Boolean and confetti percolation (see \cite{DRT17b, DRT18, GR18, M18}). I focussed on the Bernoulli model and the Poisson-Boolean model discussed in \cite{DRT17a, DRT18} to understand how the OSSS method applies to discrete and continuous models. In this document, I would like to describe these two applications in more detail. In particular, the following is meant to be an introduction to Bernoulli and Poisson-Boolean percolation at which end the OSSS method is applied. Nevertheless, this is not meant to be an overview of the historical evolution of percolation which can be found in \cite{D17a}. Also, I will only introduce the notions of percolation which are needed for the OSSS inequality. For a more complete theory, I refer to \cite{G99} for the discrete models and to \cite{MR96} for the Poisson-Boolean model. A very good starting point for percolation theory is given by the lecture notes of Duminil-Copin which can be found on his webpage.\\

The main part of the document is divided into two parts. First, I will discuss the Bernoulli (bond) percolation on $\mathbb{Z}^d$ as an introduction to percolation theory and the OSSS method. The second section is devoted to the Poisson-Boolean model. Since it heavily relies on the concept of Poisson point processes, I will give a quick introduction at the beginning of the second section.

\section{Bernoulli Percolation on $\mathbb{Z}^d$}

As mentioned before, I will only discuss bond percolation. The most apparent reason for this is that I did not extensively study the related model of \emph{site percolation}. This is partly due to the fact that it is less present in current literature and also to the fact that the two models have very similar properties. Let us start with some visualisation. The following represents the Bernoulli model on $\mathbb{Z}^2$ introduced informally in the introduction for different densities of open edges. More precisely, the images show a $40\times 40$-section of the lattice on $\mathbb{Z}^2$. For every edge, we either draw it with probability $p$ or we don't (with probability $1-p$), the drawings being independent from edge to edge.

\begin{figure}[H]
\label{p25}
\centering
\includegraphics[width=.4\textheight]{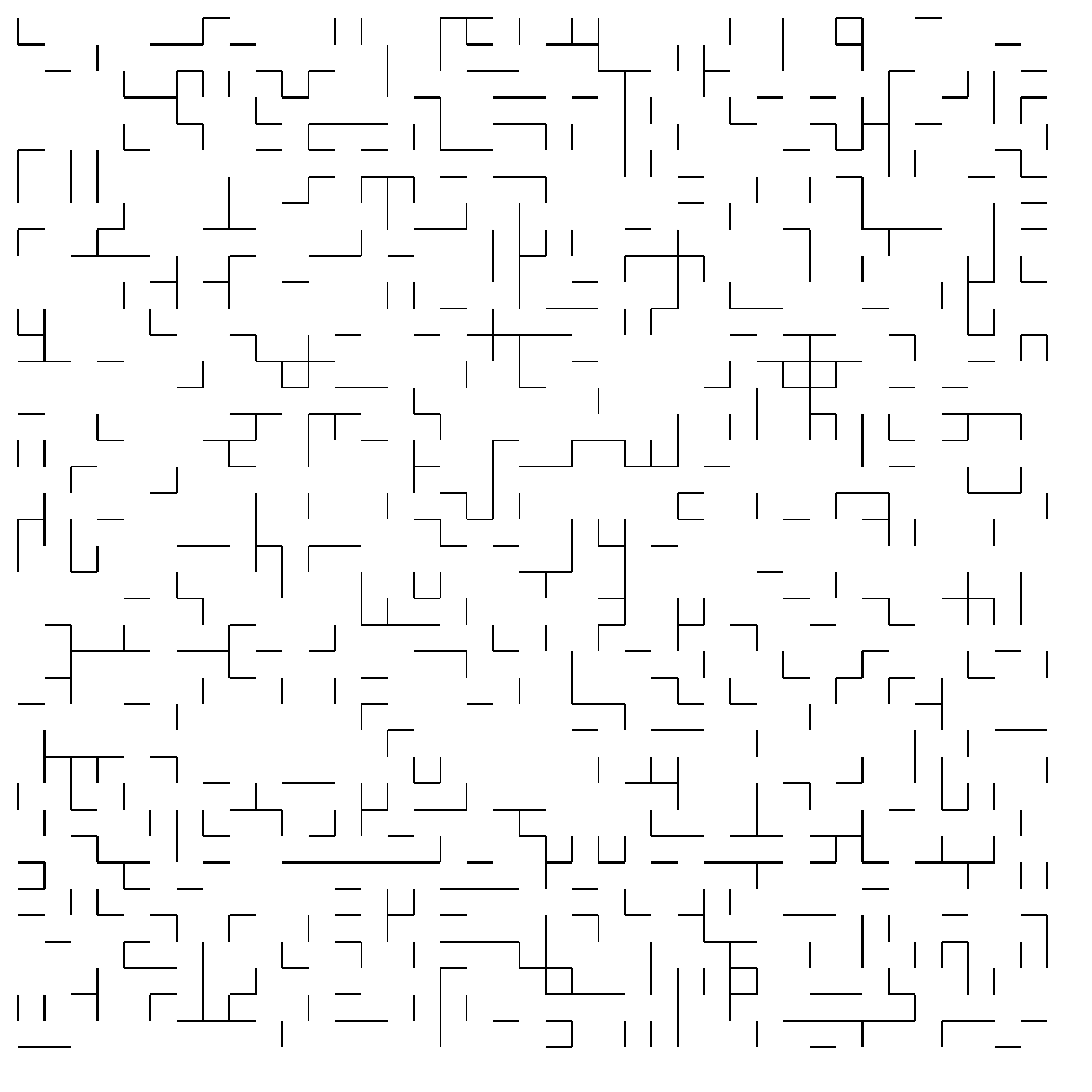}
\caption{$p = 0.25$}
\end{figure}

\begin{figure}[H]
\label{p45}
\centering
\includegraphics[width=.4\textheight]{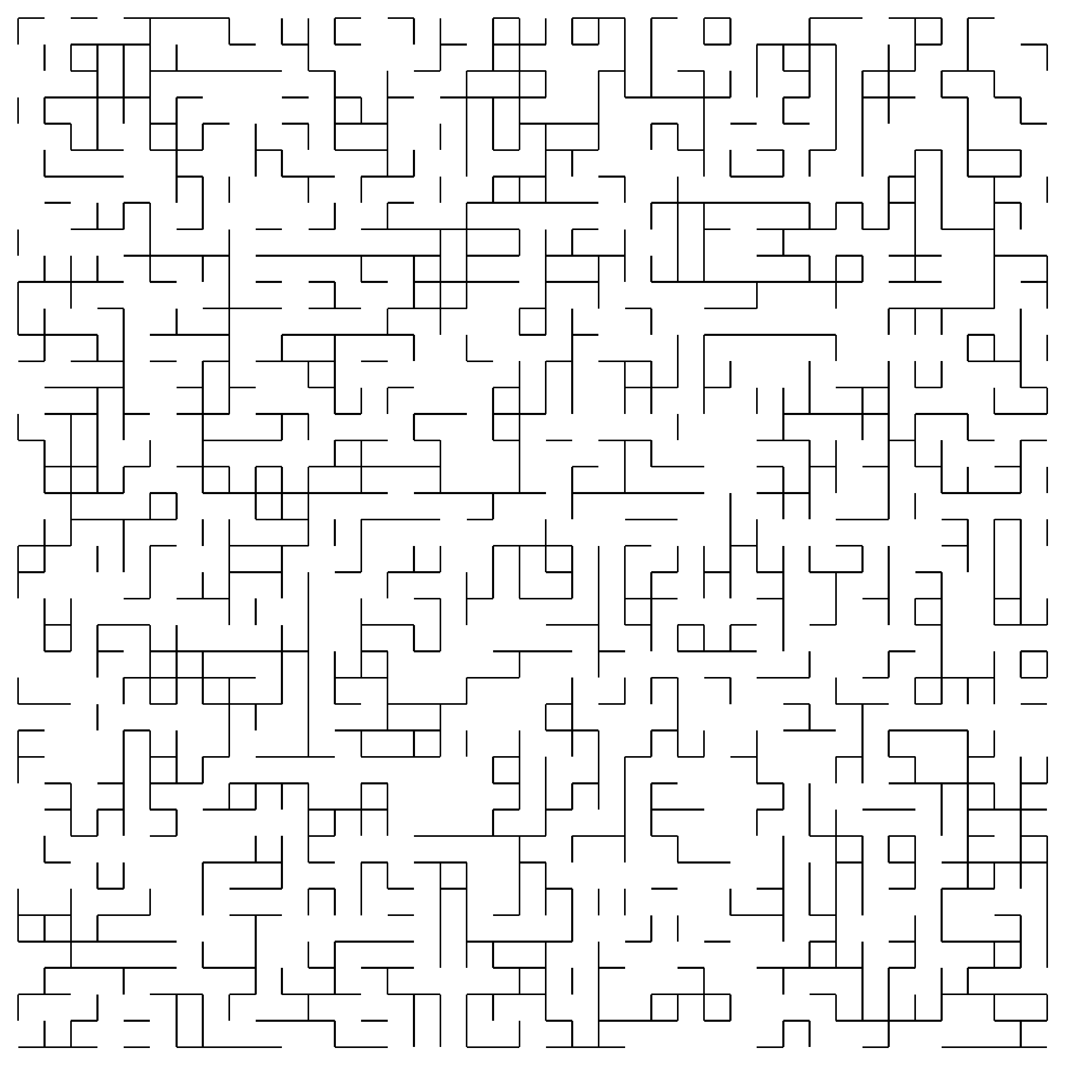}
\caption{$p = 0.45$}
\end{figure}

\begin{figure}[H]
\label{p55}
\centering
\includegraphics[width=.4\textheight]{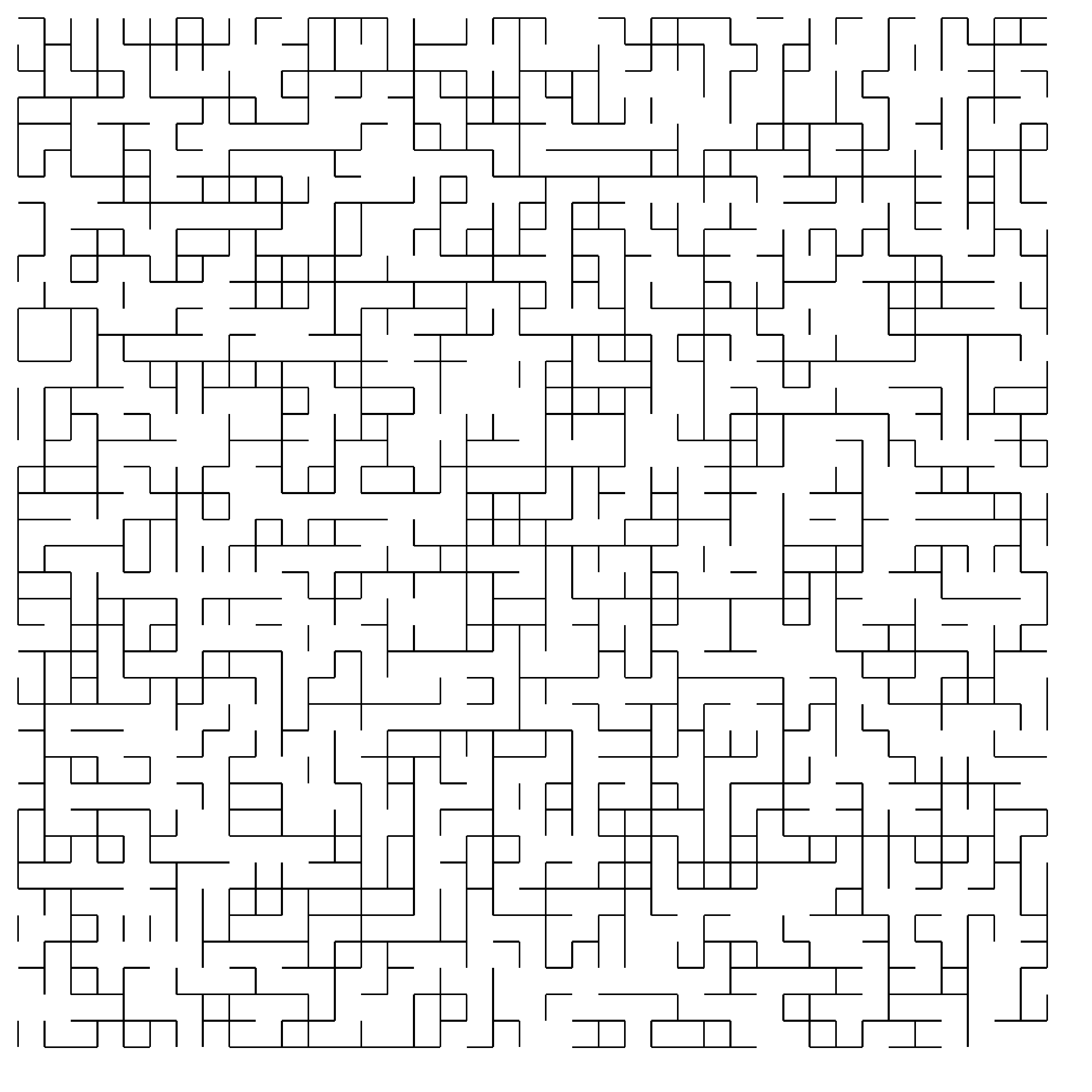}
\caption{$p = 0.55$}
\end{figure}

\begin{figure}[H]
\label{p75}
\centering
\includegraphics[width=.4\textheight]{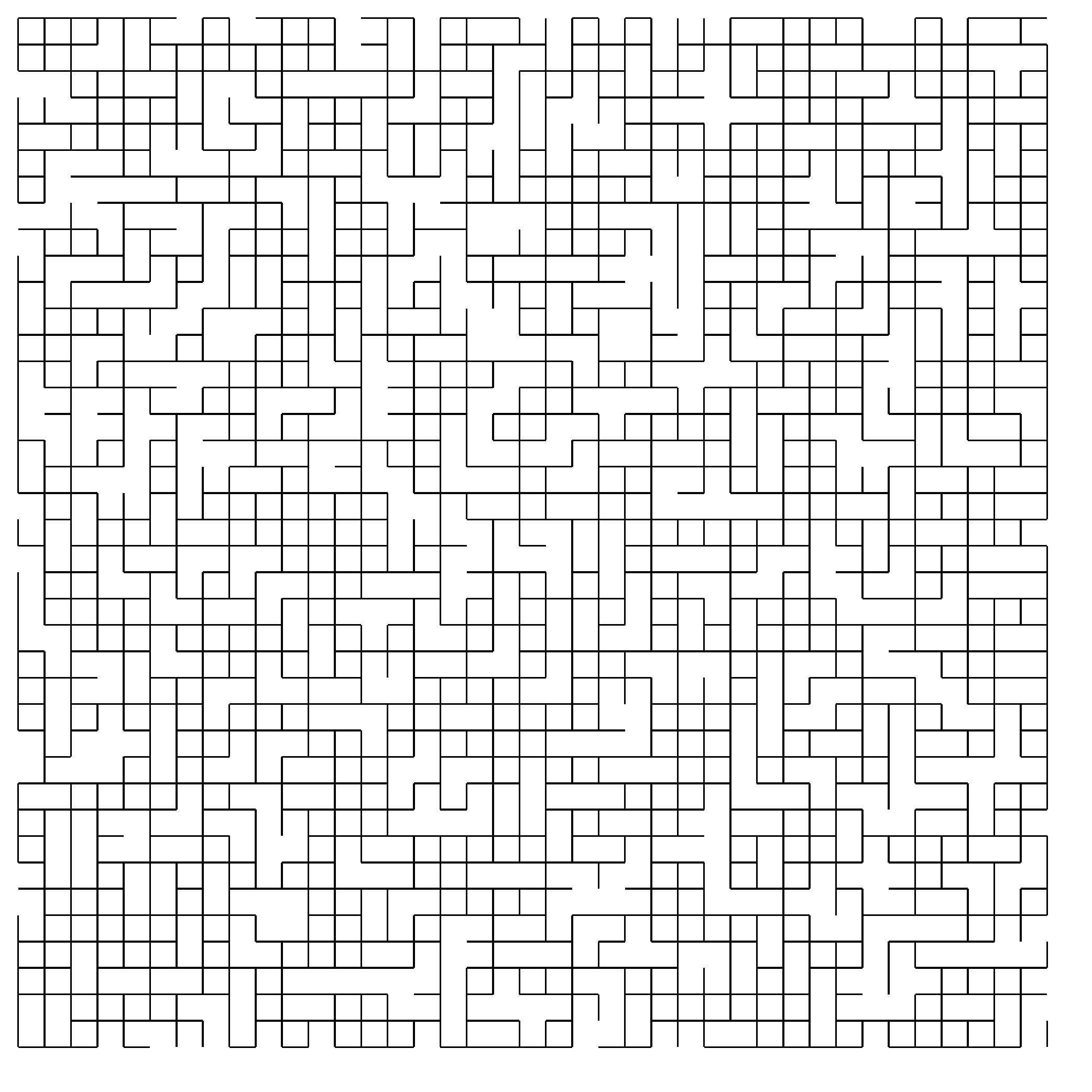}
\caption{$p = 0.75$}
\end{figure}

It is easy to see that we do not have big connected components at $p = 0.25$, also called "clusters" in this context, whereas there seems to be only one gigantic cluster at $p = 0.75$. The two other images are less clear, but if one makes the effort of looking more closely, one would see that there still is no big cluster at $p = 0.45$, but there is already only one cluster covering the entire section at $p = 0.55$. In other words, there seems to be a transition at around $p = 0.5$. Indeed, it is possible to prove that the critical density in the two-dimensional case is exactly $p_c = \frac{1}{2}$, but this issue will not concern me here. I am more interested in showing that this transition is "abrupt".

\subsection{The Model: Definitions, Notations and Basic Properties}

Consider the lattice $\mathbb{L}$ with vertices $\mathbb{Z}^d$ and with edges $E$ connecting two neighbouring points, i.e.~ for $x,y\in\mathbb{Z}^d$, we have $\{x,y\}\in E$ if and only if $\Vert x-y\Vert_1 = 1$. Let us write $\mu_p$ for the Bernoulli measure with parameter $p\in[0,1]$. We will work on the probability space given by the total space $\Omega := \{0,1\}^E$, the Borel $\sigma$-algebra $\mathcal{F}$ with respect to the product topology and the product measure $\mathbb{P}_p := \bigotimes_{e\in E}\mu_p$. 

Choosing an element $\omega\in\Omega$ of our measure space, a so-called \emph{configuration}, is nothing else than attributing to every edge $e\in E$ one of the values $0$ or $1$. To be more comfortable with this notion, let us translate this in terms of percolation: Given a configuration $\omega\in \Omega$, we will say that the edge $e\in E$ is \emph{open} if $\omega_e = 1$ and \emph{closed} if $\omega_e = 0$. This defines the random subgraph of $\mathbb{L}$ which only contains the open edges $\{e\in E\;\vert\; \omega_e = 1\}$. This subgraph corresponds to the graphs drawn above, i.e.~in which we only draw the open edges and forget the others. The connected components of this graph are called \emph{open clusters}. For $x,y\in \mathbb{Z}^d$, we say that \emph{$x$ is connected to $y$} (or: \emph{$x$ and $y$ are connected}), in symbols $x\leftrightarrow y$, if $x$ and $y$ are in the same open cluster, that is if there exists a path from $x$ to $y$ consisting of open edges only. If $x\leftrightarrow y$ for some $y\in A\subseteq\mathbb{Z}^d$, we say that \emph{$x$ is connected to $A$}, in symbols: $x\leftrightarrow A$. Finally, we say that \emph{$x$ is connected to infinity}, in symbols: $x\leftrightarrow \infty$, if $x$ lies in an infinite open cluster. If we only consider the paths on a subgraph induced on $A\subseteq\mathbb{Z}^d$, we replace $\leftrightarrow$ by $\overset{A}{\leftrightarrow}$.

Moreover, we will denote by $\Lambda_n := [-n,n]^d\cap\mathbb{Z}^d$ and $\Lambda_n(x) := x + \Lambda_n$ the boxes of size $n$. For a set $A\subseteq\mathbb{Z}^d$, we write $\partial A$ for the set of all vertices in $A$ that have a (direct) neighbour in $A^c$.\\

Let us pause a moment to get a feeling for the measure space we work in. The elements we look at are edge configurations, i.e.~subgraphs of $\mathbb{L}$. The $\sigma$-algebra of our space is generated by finite intersections of sets of the form $\{e\text{ is open}\} = \{\omega\in\Omega\;\vert\; \omega_e = 1\}\subseteq\Omega$. To be at ease with these notion, we will prove the measurability of the basic events, we defined above.

\begin{lem}
The sets $\{x\leftrightarrow y\}$, $\{x\leftrightarrow A\}$ and $\{x\leftrightarrow \infty\}$ are measurable for all $x,y\in\mathbb{Z}^d$ and $A\subseteq \mathbb{Z}^d$.
\end{lem}
\begin{proof}
The set $\{x\leftrightarrow y\}$ of all configuration for which $x$ and $y$ are connected is the set of all configuration in which there exists a (finite) path from $x$ to $y$ consisting of open edges only. In symbols, this gives
\[
\{x\leftrightarrow y\} = \bigcup_{(v_1,e_1,\dots, e_n, v_{n+1})\in \mathfrak{P}}\bigcap_{ i=1}^n \{ e_i\text{ is open }\},
\]
where $\mathfrak{P}$ denotes the set of all paths from $x$ to $y$ in $\mathbb{Z}^d$. Since the number of (finite) paths from $x$ to $y$ is countable and since the $\sigma$-algebra $\mathcal{F}$ contains all finite intersections of sets of the form $\{e\text{ is open}\}$, we conclude that $\{x\leftrightarrow y\}$ is measurable. Furthermore, the set $\{x\leftrightarrow A\}$ is a countable union of the former sets. Finally,
\[
\{x \leftrightarrow \infty\} = \bigcap_{n\in\mathbb{N}} \{x\leftrightarrow \partial\Lambda_n(x)\}
\]
is a countable intersection of measurable sets, and thus measurable itself.
\end{proof}

\begin{lem}
The number of infinite open clusters $N$ is a random variable.
\end{lem}
\begin{proof}
It suffices to prove that the set $\{N\geq k\}$ is measurable for every $k\in\mathbb{N}$. Note that $N \geq k$ if and only if there are $k$ distinct vertices $x_1,\dots,x_k\in\mathbb{Z}^d$ which are all connected to infinity, but not to each other. Thus, we may write
\[
\{ N\geq k\} = \bigcup_{\{x_1,\dots,x_k\}\in \binom{\mathbb{Z}^d}{k}} \bigcap_{i=1}^k \left(\{ x_i \leftrightarrow \infty\} \cap \bigcap_{j\neq i}      \{x_i \not\leftrightarrow x_j\}\right),
\]
where $\binom{\mathbb{Z}^d}{k}$ is the set of all $k$-subsets of $\mathbb{Z}^d$. From the previous lemma, it follows that this set is measurable for every $k\in\mathbb{N}$.
\end{proof}

Now that we have assured ourselves that the basic events really are measurable, we need to have a closer look at the probability measure.

\begin{lem}
Let $x\in\mathbb{Z}^d$. For a configuration $\omega\in\Omega$, denote by $\tau_x(\omega)$ the translated configuration defined by
\[
\tau_x(\omega)(\{v,w\}) := \omega(\{v - x, w-x\}).
\]
For an event $A\in\mathcal{F}$, define the translated event by $\tau_xA := \{ \tau_x(\omega)\;\vert\;\omega\in A\}$. Then $\tau_xA$ is an event and 
\[
\mathbb{P}_p[\tau_xA] = \mathbb{P}_p[A],
\]
i.e. $\mathbb{P}_p$ is invariant under translations.
\end{lem}
\begin{proof}
For $e = \{v,w\}\in E$ define $\tau_x e := \{v-x,w-x\}\in E$. To prove the measurability, it suffices to prove that $\tau_{x}$ is measurable for every $x\in\mathbb{Z}^d$. For this, it is sufficient to note that \[
\tau_x(\{e\text{ is open}\}) = \{ \tau_xe\text{ is open}\}\in\mathcal{F}
\]
for all $e\in E$.

To show the invariance of $\mathbb{P}_p$, recall first that the cylinder sets are a generator of $\mathcal{F}$ which is stable under finite intersection. By independence of the different edges, we only need to note that
\[
\mathbb{P}_p[\tau_xe\text{ is open}] = p = \mathbb{P}_p[e\text{ is open}].
\]
\end{proof}

We say that an event $A\in\mathcal{F}$ is \emph{translation invariant} if $\tau_xA = A$ for all $x\in\mathbb{Z}^d$. Note that the event $\{N = k\}$ is translation invariant for every $k\in\mathbb{Z}^d$. Indeed,
\begin{align*}
\tau_x\{N\geq j\} &= \bigcup_{\{x_1,\dots,x_j\}\in\binom{\mathbb{Z}^d}{j}} \bigcap_{i=1}^j \left(\tau_x\{x_i\leftrightarrow \infty\} \cap\bigcap_{l\neq i} \tau_x\{x_i\not\leftrightarrow x_l\}\right)\\
&= \bigcup_{\{x_1,\dots,x_j\}\in\binom{\mathbb{Z}^d}{j}} \bigcap_{i=1}^j \left(\{x_i-x\leftrightarrow \infty\} \cap\bigcap_{l\neq i} \{x_i-x\not\leftrightarrow x_l-x\}\right)\\
&= \bigcup_{\{x_1-x,\dots,x_j-x\}\in\binom{\mathbb{Z}^d}{j}} \bigcap_{i=1}^j \left(\{x_i\leftrightarrow \infty\} \cap\bigcap_{l\neq i} \{x_i\not\leftrightarrow x_l\}\right)\\
&= \{N\geq j\}
\end{align*}
for all $j\in\mathbb{N}$, and thus
\[
\tau_x\{N = k\} = \tau_x\{N\geq k\} \setminus \tau_x\{N\geq k+1\} = \{N\geq k\}\setminus\{N\geq k+1\} = \{N = k\}.
\]

\begin{lem}
The probability measure $\mathbb{P}_p$ is ergodic, i.e.~for all translation invariant events $A\in\mathcal{F}$, one has $\mathbb{P}_p[A] \in\{0,1\}$.
\end{lem}
\begin{proof}
Recall that we can approximate every event in $\mathcal{F}$ by events depending on a finite number of edges only. Let $(A_n)\subset\mathcal{F}$ be such an approximation. Since for every $n$, the event $A_n$ only depends on a finite number of edges, there exists $x_n\in\mathbb{Z}^d$ such that $\tau_{x_n}A_n$ is independent of $A_n$. By noting that $\tau_{x_n}A_n$ approximates $\tau_{x_n}A$, we get
\[
\mathbb{P}_p[A] = \mathbb{P}_p[A\cap \tau_{x_n}A] \approx \mathbb{P}_p[A_n\cap \tau_{x_n}A_n] = \mathbb{P}_p[A_n]\cdot\mathbb{P}_p[\tau_{x_n}A_n] = \mathbb{P}_p[A_n]^2 \approx \mathbb{P}_p[A]^2.
\]
The details in terms of some small $\epsilon$ are tedious and thus left to the reader. The equality $\mathbb{P}_p[A] = \mathbb{P}_p[A]^2$ then leads to the result.
\end{proof}

This means in particular that $N$ is almost surely constant. Hence, we need to ask, when $N\neq 0$ almost surely and which values $N$ can take. For this, we will introduce two important quantities in percolation theory. Let
\[
\theta_n(p) := \mathbb{P}_p[0\leftrightarrow \partial\Lambda_n]\quad\text{ and }\quad \theta(p) := \mathbb{P}_p[0\leftrightarrow \infty] = \lim_n \theta_n(p)
\]
be the probability that $0$ is connected to the box of size $n$ and the probability that $0$ lies in an infinite open cluster. First, note that $\theta(p) = 0$ implies $N = 0$: in fact, $N$ is nonzero if and only if there is at least one infinite open cluster, i.e.~if at least on vertex is connected to infinity. Thus, we may write
\[
\{N > 0\} = \bigcup_{x\in\mathbb{Z}^d} \{x\leftrightarrow \infty\},
\]
and hence
\[
\mathbb{P}_p[N > 0] \leq \sum_{x\in\mathbb{Z}^d} \mathbb{P}_p[x\leftrightarrow\infty] = \sum_{x\in\mathbb{Z}^d} \theta(p) = 0
\]
by translation invariance. Conversely, if $\theta(p) > 0$, then
\[
\mathbb{P}_p[N > 0] \geq \mathbb{P}_p[0\leftrightarrow\infty] = \theta(p) > 0.
\]
Hence $N > 0$ almost surely if and only if $\theta(p) > 0$. We thus are interested in the critical value, where $N$ changes from 0 to some positive value:
\[
p_c := p_c(d) := \sup \{p\in[0,1]\;\vert\; \theta(p) = 0\}.
\]

We would like to have the equivalence\[
\theta(p) > 0\qquad\text{ iff }\qquad p > p_c.
\]
For this to be true, we need $\theta$ to be increasing in $p$. This intuitive fact can be proven by using the concept of \emph{coupling}. A more detailed description of the method of coupling can be found in \cite[Section 1.3]{G99}. The following figure provides a \emph{sketch} for $\theta$. The exact form is not known for all dimensions and this figure only reflects the conjectured form.

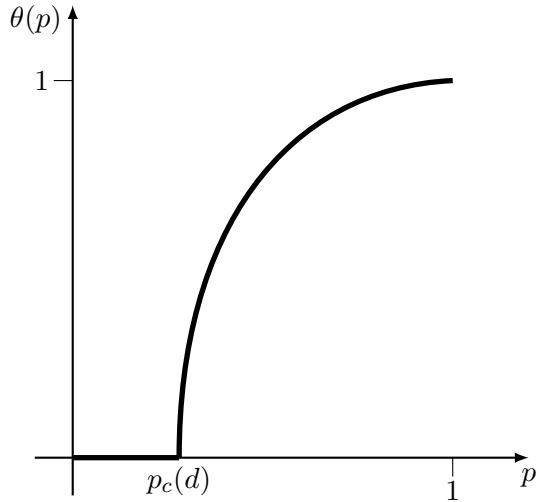
\begin{figure}[H]
\begin{center}
  \begin{tikzpicture}
    \node (1) at (1.4,-0.3) {$p_c(d)$};
    \coordinate (Origin)   at (0,0);
    \coordinate (XAxisMin) at (-0.5,0);
    \coordinate (XAxisMax) at (6,0);
    \coordinate (YAxisMin) at (0,-0.5);
    \coordinate (YAxisMax) at (0,6);
    \draw [thick, black,-latex] (XAxisMin) -- (XAxisMax)node[pos=1, below]{$p$};
    \draw [thick, black,-latex] (YAxisMin) -- (YAxisMax)node[pos=0.97, left]{$\theta (p)$};
    \draw[line width=2pt] (0, 0) .. controls(0.5,0) and (1,0) .. (1.4, 0);
  \draw[line width=2pt] (1.4, 0) .. controls(1.4,3) and (2.8,4.9) .. (5, 5);
  \draw (5,0)--(5,-0.25)node[pos=0.7, below]{1};
  \draw(0,5)--(-0.25,5)node[pos=0.7, left]{1};

  \end{tikzpicture}
  \end{center}

	\caption{Conjectured behaviour of $\theta$}
\end{figure}

One goal of percolation theory is to understand the behaviour of $\theta$, especially near the critical point. But before we get there, we first note that it can be shown that $N$ only takes the values $0$ or $1$. This means that the infinite open cluster is unique if it exists. A proof of this proposition can be found in \cite{BK89}.

In the study of the behaviour of $\theta$, one would like to show the transition from the subcritical phase to the supercritical phase to be sharp. In mathematical terms, this means that $\theta_n(p)$ converges very fast to zero for $p < p_c$. This would describe the left hand limit towards $p_c$. For a first step towards the understanding of the right side limit, we show that $\theta(p)$ grows not too slowly for $p > p_c$. Many other behaviours near the critical point are conjectured, and partially proven, but we will concentrate in this document on the first two problems. In particular, we will concentrate on a new method to show the sharpness of the transition. Before attacking this problem, we need to study two important tools in percolation theory. But first, we will have a look at the final result:

\begin{theo}
For all $p \in (0, p_c)$, there exists a constant $c_p > 0$ such that
\[
\theta_n(p) \in O_{n\to\infty}\left(e^{-c_pn}\right).
\]
Furthermore, there is a constant $c > 0$ such that
\[
\forall p > p_c,\qquad \theta(p) \geq c(p - p_c).
\]
In other words, the phase transition is sharp (exponential decay below the critical point) and $\theta$ grows at least linearly above the critical point.
\end{theo}

\subsection{The FKG Inequality and Russo's formula}

This part is dedicated to two very important concepts which are verified \emph{mutatis mutandis} in many different models. Here, we will only discuss the formulation necessary for the Bernoulli percolation model. Later on, we will see other formulations.\\

First, we will study the FKG inequality. Also known as positive correlation of the measure, it shows that a particular class of events are positively correlated. This gives some control over the intersection of dependent events. We will need some more notations to state the assertion.

On $\Omega$, we can define a partial order via $\omega \leq \omega'$ if and only if $\omega_e \leq \omega'_e$ for all $e\in E$. We then say that an event $A\in\mathcal{F}$ is \emph{increasing} if its indicator function is increasing (with respect to this partial order). In other words: if $\omega\in A$ and we add open edges, the new configuration $\omega'$ still belongs to $A$. Examples for increasing events are $\{x\leftrightarrow y\}$ or $\{N > 0\}$. We say that an event is \emph{decreasing} if its complement is increasing, i.e.~if its indicator function is decreasing. 

We are now able to state a version of the FKG inequality. It is possible to give a more general form, but the following version often suffices.

\begin{theo}[FKG inequality]
Let $f,g:\Omega \rightarrow \mathbb{R}$ be two measurable, bounded and increasing functions. Then
\[
\mathbb{E}_p[fg] \geq \mathbb{E}_p[f]\cdot\mathbb{E}_p[g].
\]
In particular, if $A,B\in\mathcal{F}$ are two increasing events, then
\[
\mathbb{P}_p[A \cap B]\geq \mathbb{P}_p[A]\cdot\mathbb{P}_p[B].
\]
The same is true for two decreasing functions or two decreasing events. If one is increasing and the other decreasing, the inequality is reversed.
\end{theo}
\begin{proof}
We follow the proof from \cite[Theorem 2.4]{G99}. Since the second statement is a particular case of the first one, it suffices to show the FKG inequality for expectations. In the following, we will use $e$ and $\omega_e$, i.e. an edge and its value, interchangeably to simplify the notations.

First assume that $f$ and $g$ only depend on a finite number of edges. We will prove the statement by induction. If $f$ and $g$ only depend on one edge, we have
\begin{align*}
\mathbb{E}_p[fg] - \mathbb{E}_p[f]\mathbb{E}_p[g] &= p(1-p)\Big{(}f(1)g(1) -f(0)g(1) - f(1)g(0)+ f(0)g(0)\Big{)}\\
&= p(1-p) \underbrace{\Big{(} f(1) - f(0)\Big{)}}_{\geq 0}\underbrace{\Big{(} g(1) - g(0)\Big{)}}_{\geq 0} \geq 0.
\end{align*}
Now, suppose that the assertion is true for all pairs of functions depending on $n-1$ edges for some $n\geq 2$ and consider functions $f$ and $g$ depending on $n$ edges $e_1,\dots e_n$. In particular, $\mathbb{E}_p[f\;\vert\; e_1,\dots,e_{n-1}]$ and $\mathbb{E}_p[g\;\vert\; e_1,\dots,e_{n-1}]$ depend on the $n-1$ edges $e_1,\dots, e_{n-1}$ only. Furthermore, the conditioning preserves the monotony of $f$ and $g$. This then gives
\begin{align*}
\mathbb{E}_p[fg] = \mathbb{E}_p\Big{[}\mathbb{E}_p[fg\;\vert\; e_1,\dots,e_{n-1}]\Big{]} &\geq  \mathbb{E}_p\Big{[} \mathbb{E}_p[f\;\vert\; e_1,\dots,e_{n-1}]\cdot\mathbb{E}_p[g\;\vert\; e_1,\dots,e_{n-1}]\Big{]} \\
&\geq \mathbb{E}_p\Big{[}\mathbb{E}_p[f\;\vert\; e_1,\dots,e_{n-1}]\Big{]}\cdot\mathbb{E}_p\Big{[}\mathbb{E}_p[g\;\vert\; e_1,\dots,e_{n-1}]\Big{]} \\
&= \mathbb{E}_p[f]\cdot\mathbb{E}_p[g],
\end{align*}
where we used in the first inequality that the conditional expectation can be seen as a finite sum of regular expectations on the events that the edges $e_1,\dots,e_{n-1}$ have been chosen. Hence, the base case applies. Then, the second inequality is due to the induction hypothesis.

Finally, consider two bounded increasing random variables $f$ ans $g$. Take an ordering $(e_n)_{n\geq 1}$ of the edge set $E$ and consider the two martingales defined by $f_n := \mathbb{E}_p[f\;\vert\; e_1,\dots,e_n]$ and $g_n :=  \mathbb{E}_p[g\;\vert\; e_1,\dots,e_n]$. Then $f_n$ and $g_n$ only depend on the edges $e_1,\dots,e_n$. We conclude that
\[
\mathbb{E}_p[f_ng_n]\geq \mathbb{E}_p[f_n]\cdot\mathbb{E}_p[g_n]
\]
for all $n\geq 1$. Finally, the martingale convergence theorem yields
\[
\mathbb{E}_p[fg] = \lim_n \mathbb{E}_p[f_ng_n] \geq \left(\lim_n \mathbb{E}_p[f_n]\right)\cdot\left(\lim_n \mathbb{E}_p[g_n]\right) = \mathbb{E}_p[f]\cdot\mathbb{E}_p[g].
\]
\end{proof}

I will give a short example to illustrate the intuition behind the FKG inequality.  Take the two events $\{v_1\leftrightarrow v_2\}$ and $\{w_1\leftrightarrow w_2\}$. The FKG inequality tells us that one event makes the other more probable:
\[
\mathbb{P}_p[v_1\leftrightarrow v_2\;\vert\; w_1\leftrightarrow w_2] \geq \mathbb{P}_p[v_1\leftrightarrow v_2]
\]
and vice versa. This is not surprising, because we already have many open edges on the event $\{w_1\leftrightarrow w_2\}$ which we can use to construct the path from $v_1$ to $v_2$.

Sometimes one would like to have the converse: on the event $A$ it is harder to obtain $B$. One way is to use the FKG inequality for an increasing and a decreasing event. But the FKG inequality tells us that this is impossible for two increasing events. To get around this problem, one adds an additional constraint: it is more difficult to obtain $A$ and $B$ in a disjoint way than to get $A$ and $B$ independently. This concept is known as the BK inequality. Since we do not need it subsequently, I will not go into the details what is meant by "in a disjoint way". For more details, I refer to \cite[Section 2.3]{G99}.\\

The second important concept concerns the characterization of $\theta_n$. This will help us to gather information about $\theta$ as limit function. The idea is that $\theta_n$ is a polynomial in $p$, which means that we can calculate the derivative explicitly. This will not be true any more for $\theta$ (which is not differentiable at the critical point). In the following, we will give two characterizations of $\theta_n'$. Eventually, we will only need the second one for the Bernoulli model. For other models however, we will use variants of the first formula.

Fix a configuration $\omega\in\Omega$. In the following, we will say that an edge $e$ is \emph{pivotal} for $A$ in $\omega$, if $\omega\in A$ and $\omega'\not\in A$, where $\omega'\in\Omega$ is the configuration equal to $\omega$ on all edges $f\neq e$ and such that $\omega'_e = 1 - \omega_e$. In other words, pivotal edges are crucial for an event to occur in a certain configuration. 

\begin{theo}[Russo's formula]
Let $A$ be an increasing event which depends on a finite number of edges only. Then
\[
\dfrac{d}{dp}\mathbb{P}_p[A] = \sum_{e\in E} \mathbb{P}_p[e\text{ is pivotal for }A].
\]
\end{theo}
\begin{proof}
We follow the general structure of the proof in \cite[Theorem 2.25]{G99}. First we couple the percolation models with respect to the parameters $p$ and $p+\delta$ with $\delta > 0$ small. (Since the event depends on a finite number of edges only, the function to consider is a polynomial. Hence, it suffices to prove the result for the right hand differential.) We do this as follows. Let $(U_e)_{e\in E}$ be an iid sequence of uniform random variables in $[0,1]$, and denote by $P$ the corresponding distribution. We define the two configurations $\omega$ and $\omega'$ via
\[
\omega(e) = 1_{U_e < p}\quad\text{ and }\quad \omega'(e) = 1_{U_e < p + \delta}.
\]
Then $\mathbb{P}_p[A] = P(\omega\in A)$ and $\mathbb{P}_{p+\delta}[A] = P(\omega'\in A)$. Suppose that $A$ depends on the edges $e_1,\dots,e_n$. Then, define the configurations $\omega_0,\dots,\omega_n\in\Omega$ inductively via $\omega_0 := \omega$,
\[
\forall f\neq e_{i+1},\quad\omega_{i+1}(f) = \omega_i(f)\quad\text{ and }\quad \omega_{i+1}(e_{i+1}) = 1_{U_{e_{i+1}} < p+\delta}
\]
for all $i = 0,\dots, n-1$. Since $A$ only depends on these $n$ edges, we obtain the equality $P(\omega'\in A) = P(\omega_n\in A)$. Since $A$ is increasing, this leads to
\begin{align*}
\mathbb{P}_{p+\delta}[A] - \mathbb{P}_p[A] &= P(\omega'\in A, \omega\not\in A)\\
&= \sum_{i=1}^n P(\omega_i\in A, \omega_{i-1}\not\in A)\\
&= \sum_{i=1}^n P(\omega_i(e_i) = 1, \omega_{i-1}(e_i) = 0, e_i\text{ is pivotal for }(A,\omega_i))\\
&= \sum_{i=1}^n P(U_{e_i} \in [p, p+\delta), e_i\text{ is pivotal for }(A,\omega_i)).
\end{align*}
Note that the event $\{e_i\text{ is pivotal for }(A,\omega_i)\}$ only depends on the edges $E\setminus\{e_i\}$. Hence,
\[
\mathbb{P}_{p+\delta}[A] - \mathbb{P}_p[A] = \delta\sum_{i=1}^n P(e_i\text{ is pivotal for }(A,\omega_i)).
\]
To finish the proof, recall that $A$ only depends on a finite number of edges. This means that the probabilities on the right hand side are polynomials in $p$ and $p+\delta$, giving
\[
P(e_i\text{ is pivotal for }(A,\omega_i)) \underset{\delta\downarrow 0}{\longrightarrow} \mathbb{P}_p[e_i\text{ is pivotal for A}]
\]
by continuity. Finally, $\mathbb{P}_p[e\text{ is pivotal for }A] = 0$ for all $e\in E\setminus\{e_1,\dots,e_n\}$. 

Note that we proved the result only for $p < 1$. For $p = 1$, this is true by continuity.
\end{proof}

We will now show a more straight forward way to characterise $\theta_n'$. This will also be the formula we will use for the Bernoulli percolation. Note, that the following characterisation does not depend on whether the event is monotone or not.

\begin{prop}
Let $A$ be an event depending on a finite number of edges only. Then
\[
\dfrac{d}{dp}\mathbb{P}_p[A] = \dfrac{1}{p(1-p)}\sum_{e\in E} \mathrm{Cov}_p(\omega_e, 1_A).
\]
\end{prop}
\begin{proof}
We follow the proof from \cite[Theorem 2.34]{G99}. Suppose that $A$ depends on the edges on a finite set $\mathcal{E}$. In particular, we can shift the problem onto a finite state space $\widetilde{\Omega} = \{0,1\}^\mathcal{E}$. For simplicity, we will still denote the corresponding probability measure by $\mathbb{P}_p$. Now, let $n:\widetilde{\Omega}\rightarrow\mathbb{N}$ be the number of open edges and let $m = \vert \mathcal{E}\vert$ be the total number of edges. Then
\[
\mathbb{P}_p[A] = \sum_{\omega\in\widetilde{\Omega}} 1_A(\omega)p^{n(\omega)}(1-p)^{m - n(\omega)}.
\]
Since the sum is finite, we may differentiate term by term, giving
\begin{align*}
\dfrac{d}{dp}\mathbb{P}_p[A] &= \sum_{\omega\in\widetilde{\Omega}} 1_A(\omega)p^{n(\omega)}(1-p)^{m - n(\omega)}\cdot\left(\dfrac{n(\omega)}{p} - \dfrac{m - n(\omega)}{1 - p}\right)\\
&= \dfrac{1}{p(1-p)} \sum_{\omega\in\widetilde{\Omega}} 1_A(\omega)p^{n(\omega)}(1-p)^{m - n(\omega)}\cdot\Big{(} n(\omega) - mp\Big{)}\\
&= \dfrac{1}{p(1-p)} \sum_{e\in \mathcal{E}} \sum_{\omega\in\widetilde{\Omega}} 1_A(\omega)p^{n(\omega)}(1-p)^{m - n(\omega)}\cdot \Big{(}\omega_e - p\Big{)}\\
&= \dfrac{1}{p(1-p)} \sum_{e\in \mathcal{E}} \mathbb{E}_p\left[ \Big{(} \omega_e - p\Big{)}\cdot 1_A\right]\\
&= \dfrac{1}{p(1-p)} \sum_{e\in\mathcal{E}} \mathrm{Cov}_p(\omega_e, 1_A)\\
&= \dfrac{1}{p(1-p)} \sum_{e\in E} \mathrm{Cov}_p(\omega_e, 1_A).
\end{align*}
\end{proof}

We conclude this section with a very important note: In the two proofs, we did not need any information on the underlying graph. This means that both formulas are true for Bernoulli percolation on any (locally finite) graph.

\subsection{The OSSS Method}

We will now introduce the OSSS method. The entire section is inspired by \cite{DRT17a} and a lecture from Duminil-Copin which can be found on YouTube under the name "Sharp threshold phenomena in Statistical Physics" \cite{D17b}. The aim is to prove the sharpness of the phase transition. As mentioned in the previous section, we will use some sort of differential inequality for $\theta_n$ to deduce information on $\theta$. More precisely, we will prove for some $c > 0$ that
\begin{equation}\label{eq:diff_n}
\theta_n' \geq c\dfrac{n}{\Sigma_n}\theta_n,
\end{equation}
where $\Sigma_n := \sum_{k=0}^{n-1} \theta_k$.

\begin{lem}\label{lem:diff_ineq}
Assume that a family of increasing and differentiable functions $(f_n)$ defined on some interval $[0,b]$ satisfies \eqref{eq:diff_n} (where we replace $\theta_n$ by $f_n$) for some $c >0$. Assume furthermore that the sequence converges pointwise to some function $f$. Then there exists a point $p_c\in [0,b]$ such that we have the two properties:
\begin{enumerate}
	\item For every $p < p_c$, there exists some constant $c_p > 0$ such that
	\[
	f_n(p) = O_{n\to\infty}\left(\exp(-c_pn)\right).
	\]
	\item For every $p > p_c$, one has $f(p) \geq c(p - p_c)$.
\end{enumerate}
\end{lem}
\begin{proof}
We follow the proof from \cite[Lemma 3.1]{DRT17a}. First note that we can assume $c = 1$. Indeed, we may simply consider the functions $(f_n(\cdot/c))$.

Now, define $$\beta_c := \inf\llbrace \beta\in [0,b]\;\vert\;\limsup_n \dfrac{\log \Sigma_n(\beta)}{\log n} \geq 1\rrbrace.$$
To prove the assertion, we will prove the points 1) and 2) interchanging $p_c$ by $\beta_c$. By definition of $p_c$, we then immediately get $p_c = \beta_c$. In the following, we will use that $\Sigma_n$ is non decreasing without mentioning it explicitly.
\begin{enumerate}
	\item For the first point consider some point $\beta < \beta_c$ and take $\delta > 0$ small enough so that $\beta + 2\delta < \beta_c$. By definition of $\beta_c$, there exists $\alpha > 0$ such that $\Sigma_n(\gamma) \leq n^{1-\alpha}$ for all $\gamma\in [0, \beta+2\delta]$, for almost all $n\in\mathbb{N}$. Fix such an $n$. Then, the inequality \eqref{eq:diff_n} becomes $f_n' \geq n^\alpha f_n$. Integrating this new inequality between $\beta+\delta$ and $\beta+2\delta$ yields
	\[
	\ln\big{(}f_n(\beta + 2\delta)\big{)} - \ln\big{(} f_n(\beta + \delta)\big{)} = \int_{\beta+\delta}^{\beta + 2\delta} \dfrac{f_n'(x)}{f_n(x)}\dx \geq \delta n^\alpha,
	\]
	i.e.
	\[
	f_n(\beta+ \delta) \leq f_n(\beta + 2\delta)\exp(-\delta n^\alpha) \leq \exp(-\delta n^\alpha)
	\]
	for almost all $n\in\mathbb{N}$. This means in particular that the series $(\Sigma_n(\beta +\delta))_n$ converges. Let $\Sigma > 0$ such that $\Sigma_n(\gamma) \leq \Sigma$ for all $\gamma\in [0,\beta+\delta]$ and for all $n$. This leads to the equation $f_n' \geq \dfrac{n}{\Sigma}f_n$. Integrating this inequality from $\beta$ to $\beta + \delta$ finally yields as before
	\[
	f_n(\beta) \leq f_n(\beta + \delta)\exp(-\delta n/\Sigma) \leq \exp(-c_\beta n)
	\]
	where $c_\beta := \delta/\Sigma$.
	
	\item Now, take $\beta > \beta_c$. For $n > 1$, define\[
	T_n := \dfrac{1}{\ln n}\sum_{i=1}^n \dfrac{f_i}{i}.
	\]
	With \eqref{eq:diff_n}, we obtain the differential inequality
	\[
	T_n' \geq \dfrac{1}{\ln n}\sum_{i=1}^n \dfrac{f_i}{\Sigma_i}.
	\]
	Hence, the fact that
	\[
	\dfrac{f_i}{\Sigma_i} \geq \int_{\Sigma_i}^{\Sigma_{i+1}} \dfrac{1}{t}\d{t} = \ln(\Sigma_{i+1}) - \ln (\Sigma_i)
	\]	
	leads to
	\[
	T_n' \geq \dfrac{\ln \Sigma_{n+1} - \ln\Sigma_1}{\ln n} \geq \dfrac{\ln\Sigma_{n+1}}{\ln n} \geq \dfrac{\ln\Sigma_n}{\ln n},
	\]
	where we used $\Sigma_1 = f_0 = 1$. Integrating this inequality between $\beta'\in (\beta_c,\beta)$ and $\beta$ yields $$
	T_n(\beta) \geq T_n(\beta) - T_n(\beta') \geq (\beta - \beta')\cdot \dfrac{\ln\Sigma_n}{\ln n}.
	$$
	Using the pointwise limit $f = \lim_n T_n$, we thus obtain
	\[
	f(\beta) \geq (\beta- \beta')\cdot \limsup_n \dfrac{\ln\Sigma_n}{\ln n} \geq \beta - \beta'.
	\]
	The assertion then follows by letting $\beta'$ tend to $\beta_c$.
\end{enumerate}
\end{proof}

Hence, the difficult work is to obtain the differential inequality \eqref{eq:diff_n}. That is where the OSSS method comes into play. The novelty of the method is to introduce the language of random algorithms. Since this is not the main part of this work, I will keep the definition as informal as possible. Interested readers may refer to \cite{OSSS05}.

Let $I$ be a finite index set. Suppose that the state space is given by $\Omega := \prod_{i\in I} \Omega_i$ for some arbitrary spaces $(\Omega_i,\mathcal{F}_i)$. Now, consider a function $f: \Omega \rightarrow [0,1]$ that you would like to compute. Fix a configuration $\omega\in\Omega$. An \emph{algorithm $T$ determining $f$} will compute $f(\omega) = f(\omega_i\;\vert\; i\in I)$ by querying one index after another, stopping as soon as $f(\omega)$ is determined. More formally, $T$ is given by an initial index $i_1\in I$ and a family of decision rules $(\phi_t)_{t\geq 1}$. These decision rules will decide which index the algorithm will reveal next, given the information it has gather until then: 
$$i_{t+1} = \phi_t(i_1,\dots,i_{t}, \omega(i_1),\dots,\omega(i_{t})).$$ We say that $T$ stops at instant $t\geq 1$ if the outcome of $f$ only depends on the revealed values $i_1,\dots, i_{t}$ and $\omega(i_1),\dots,\omega(i_{t})$.

Let us have a look at a trivial example. Take $\Omega = \mathbb{R}^n$ and $f(x_1,\dots,x_n) = \vert x_1\vert \wedge 1$. Consider two algorithms $T_1$ and $T_2$ which will proceed as follows: $T_1$ starts with index $i_1= 1$ and the next index will always be the following one, i.e. $i_2 = 2$, $i_3 = 3$ etc. Since $f$ only depends on its first coordinate, $T_1$ will stop at time $1$. The second algorithm starts with index $i_1 = n$ and goes backwards through the indices. In particular, it will reveal $x_1$ only at step $t = n$, hence it only stops at time $n$.

In the setting of Bernoulli percolation, we might take $\Omega = \{0,1\}^E$ with the slight difficulty that the edge set $E$ is not finite. But this can be circumvented easily. The function we want to compute is $f(\omega) = f(\omega_e\;\vert\; e\in E) := 1_{0\leftrightarrow \partial\Lambda_n}(\omega)$, since $$\theta_n(p) = \mathbb{P}[0\leftrightarrow\partial\Lambda_n] = \mathbb{E}[f].$$ One trivial algorithm to determine $f$ would be to put an arbitrary order on $E$ and to query one edge after the other (following this order) until we know whether $0$ is connected to $\partial\Lambda_n$ or not. Obviously, that is not the most intelligent way and we might end up waiting $\Omega(n^d)$ steps until the algorithm stops. Later on, we will see how to construct more effective algorithms. For now, let us come back to the general setting, and some more definitions.

Now suppose that the spaces $(\Omega_i)_i$ are equipped with the structure of a measure space and equip $\Omega$ with the product structure. We suppose that the considered applications are all measurable. Then $T$ can be understood as a random variable from $\Omega$ to $I^{\vert I\vert}$. If $P$ denotes the product measure, we then define the \emph{revealment} of an index $i\in I$ as
\[
\delta_i(T) = P(\text{$T$ reveals $i$}),
\]
i.e.~the probability that $T$ does not stop before revealing $i$. Furthermore, we define the \emph{influence} of an index $i$ on the function $f$ by
\[
\mathrm{Inf}_i(f) = P(f(\omega) \neq f(\omega')),
\]
where $\omega'$ is the configuration equal to $\omega$ for all indices $j\neq i$ and where $\omega'(i)$ is resampled independently with respect to the distribution on $\Omega_i$. (We simplify the notation by denoting the resulting distribution also by $P$.)\\

The OSSS method relies on the following result from \cite{OSSS05}. As for the previous section, we will present two versions of the OSSS inequality. For the Bernoulli percolation, we only need the second version.
\begin{theo}[OSSS inequality, see also \cite{OSSS05}]\label{theo:OSSS}
Take the notation from above. Then
\[
\mathrm{Var}(f) \leq \sum_{i\in I} \delta_i(T)\mathrm{Inf}_i(f).
\]
Suppose now that $\Omega = \{0,1\}^I$ for some finite $I$, equipped with the product of Bernoulli measures. If $f$ is non decreasing, we also obtain
\[
\mathrm{Var}(f) \leq 2\sum_{i\in I} \delta_i(T)\mathrm{Cov}(\omega_i, f).
\]
In the above, $\mathrm{Var}$ and $\mathrm{Cov}$ respectively denote the variance and the covariance with respect to $P$. 
\end{theo}

However, before attacking the proof, we will first have a look at the general idea. For this, note that
\[
\mathrm{Var}(f) = E\big{(} (f - E(f))^2\big{)} \leq E(\vert f - E(f)\vert),
\]
since $f$ takes its values in $[0,1]$. Here and in the following, $E$ denotes the expectation with respect to $P$. We will now take two iid samples $\omega,\widetilde{\omega}\in\Omega$ with respect to $P$. Rewrite the right hand side from the above inequality as
\begin{align*}
E(\vert f - E(f)\vert) = \mathbb{E}\Big{(} \vert f(\omega) - E(f)\vert\Big{)} &=
 \mathbb{E}\Big{(} \big{\vert} \mathbb{E}(f(\omega)\;\vert\; \omega) - \mathbb{E}
 (f(\widetilde{\omega})\;\vert\;\omega)\big{\vert}\Big{)} \\
 &\leq \mathbb{E}\Big{(} \mathbb{E}(\vert f(\omega) - f(\widetilde{\omega})\vert\;\vert\;\omega)\Big{)}\\
 &= \mathbb{E}(\vert f(\omega) - f(\widetilde{\omega})\vert),
\end{align*}
where $\mathbb{E}$ denotes the expectation with respect to the distribution of $(\omega,\widetilde{\omega})$. We will then define "intermediate" states $\omega^{(i)}$ which will split the dependencies with respect to the single indices. 

\begin{proof}[Proof of Theorem \ref{theo:OSSS}]
We follow the proof from \cite{OSSS05}. Denote by $i_1,\dots,i_\tau$ the indices queried by $T$ from $\omega$ before stopping. Note that $\tau$ is a random variable! Then write
\[
J[t] := \{ i_r\;\vert\; t < r \leq \tau\}
\]
for the indices queried after the instant $t$. We then define the configuration $\omega^{(t)}$ as the configuration which equals $\omega$ on $J[t]$ and which equals $\widetilde{\omega}$ on $I\setminus J[t]$. In particular, $f\left(\omega^{(0)}\right) = f\left(\omega\right)$ and $f\left(\omega^{(\tau)}\right) = f\left(\widetilde{\omega}\right)$. This yields
\begin{align*}
\mathrm{Var}(f) &\leq \mathbb{E}(\vert f(\omega) - f(\widetilde{\omega})\vert)\\
&\leq \mathbb{E}\left(\sum_{t=1}^\tau \vert f(\omega^{(t-1)}) - f(\omega^{(t)})\vert\right)\\
&= \sum_{t=1}^{\vert I\vert} \mathbb{E}\left( \vert f(\omega^{(t-1)}) - f(\omega^{(t)})\vert \cdot 1_{t\leq \tau}\right).
\end{align*}
If we set $i_t$ to some fixed value which is \emph{not} in $I$ for all $t > \tau$, then we get $1_{t\leq \tau} = \sum_{i\in I} 1_{i_t = i}$, and thus
\[
\mathrm{Var}(f) \leq \sum_{i\in I} \sum_{t = 1}^{\vert I\vert} \mathbb{E}\left(\vert f(\omega^{(t-1)}) - f(\omega^{(t)})\vert\cdot 1_{i_t = i}\right).
\]
Now, denote by $X_t$ the values revealed by $T$ until the step $t$, i.e.
\[
X_t = (\omega(i_1),\dots, \omega(i_{t\wedge\tau})).
\]
Note that $1_{i_t = i}$ is by definition $X_{t-1}$-measurable. Furthermore $\omega^{(t-1)}$ and $\omega^{(t)}$ only differ in the $i_t$-th coordinate. Using the fact that $\widetilde{\omega}$ and $(\omega_i)_{i\not\in\{i_1,\dots,i_{(t-1)\wedge\tau}\}}$ are independent of $X_{t-1}$, it follows that $\omega^{(t-1)}$ and $\omega^{(t)}$ are independent from $X_{t-1}$, hence
\begin{align*}
\mathbb{E}\left(\vert f(\omega^{(t-1)}) - f(\omega^{(t)})\vert\cdot 1_{i_t = i}\right) &= \mathbb{E}\left( \mathbb{E}\left(\vert f(\omega^{(t-1)}) - f(\omega^{(t)})\vert\;\Big{\vert}\; X_{t-1}\right)\cdot 1_{i_t = i}\right)\\
&= \mathbb{E}\left(\mathbb{E}\left(\vert f(\omega^{(t-1)}) - f(\omega^{(t)})\vert\right)\cdot 1_{i_t = i}\right).
\end{align*}
Now, $\vert f(\omega^{(t-1)}) - f(\omega^{(t)})\vert \leq 1_{f(\omega^{(t-1)}) \neq f(\omega^{(t)})}$ and we can conclude that
\begin{align*}
\mathrm{Var}(f) &\leq \sum_{i\in I}\sum_{t= 1}^{\vert I\vert} \mathbb{E}\left(\mathrm{Inf}_{i_t}(f)\cdot 1_{i_t = i}\right)\\
&= \sum_{i\in I} \mathrm{Inf}_i(f) \cdot\sum_{t=1}^{\vert I\vert} \mathbb{P}(i_t = i)\\
&= \sum_{i\in I} \mathrm{Inf}_i(f) \cdot \mathbb{P}(i_t = i\text{ for some }t\leq \tau)\\
&= \sum_{i\in I} \delta_i(T)\mathrm{Inf}_i(f).
\end{align*}

Let us now prove the second version: Assume that we have $\Omega = \{0,1\}^I$ with the product Bernoulli measure and assume $f$ to be non decreasing. Then, on the event $\{i_t = i\}$,
\begin{align*}
\vert f(\omega^{(t-1)}) - f(\omega^{(t)})\vert &= \left(f(\omega^{(t-1)}) - f(\omega^{(t)})\right)\cdot\left(\omega_{i}^{(t-1)} - \omega_i^{(t)}\right)\\
&= f(\omega^{(t-1)})\omega_i^{(t-1)} + f(\omega^{(t)})\omega_i^{(t)} - f(\omega^{(t-1)})\omega_i^{(t)} - f(\omega^{(t)})\omega_i^{(t-1)}.
\end{align*}
As before, we get by independence
\[
\mathbb{E}\left( f(\omega^{(t-1)})\omega_i^{(t-1)}\;\vert\; X_{t-1}\right) = \mathbb{E}[f(\omega)\omega_i] = \mathbb{E}\left(f(\omega^{(t)})\omega_i^{(t)}\;\vert\; X_{t-1}\right).
\]
(For the second equality, we use the tower property and the fact that $\sigma(X_{t-1})\subseteq \sigma(X_t)$.) If we could show on $\{i_t = i\}$ both
\[
\mathbb{E}\left(f(\omega^{(t-1)})\omega_i^{(t)}\;\vert\; X_{t-1}\right) \geq E(f)\mathbb{E}(\omega_i)
\]
and
\[
\mathbb{E}\left(f(\omega^{(t)})\omega_i^{(t-1)}\;\vert\; X_{t-1}\right) \geq E(f)\mathbb{E}\left(\omega_i\right),
\]
it would follow that
\begin{align*}
\mathrm{Var}(f) &\leq \sum_{i\in I}\sum_{t=1}^{\vert I\vert} \mathbb{E}\Big{(}\big{(}2\mathbb{E}(f(\omega)\omega_i) - 2E(f)\mathbb{E}(\omega_i)\big{)}\cdot 1_{i_t = i}\Big{)}\\
&= \sum_{i\in I} 2\mathrm{Cov}(f, \omega_i)\cdot \sum_{t = 1}^{\vert I\vert} \mathbb{P}(i_t = i)\\
&= 2\sum_{i\in I} \delta_i(T)\mathrm{Cov}(f,\omega_i).
\end{align*}
Let us have a closer look at the two inequalities we need to prove. Since $f$ is non decreasing, for fixed $\omega$, $f(\omega^{(t-1)})\omega_i^{(t)}$ is non decreasing in $\widetilde{\omega}$. With the FKG inequality, this gives us
\[
\mathbb{E}\left( f(\omega^{(t-1)})\omega_i^{(t)}\;\vert\; X_{\vert I\vert}\right) \geq \mathbb{E}\left(f(\omega^{(t-1)})\;\vert\; X_{\vert I\vert}\right)\cdot\mathbb{E}\left(\omega_i^{(t)}\;\vert\; X_{\vert I\vert}\right).
\]
Thus, the first lower bound follows from the tower property for conditional expectations, the fact that $\mathbb{E}\left(\omega_i^{(t)}\;\vert\; X_{\vert I\vert}\right)$ is $X_{t-1}$ measurable and the independence of $\omega^{(t-1)}$ and $\omega^{(t)}$ from $X_{t-1}$. We get a similar result for the second term, but only with respect to the conditioning on $X_t$. The final result then follows from another application from the tower property and said independence.
\end{proof}

The second version is very useful for Bernoulli percolation, because it is directly related to the derivative of $\theta_n$. To conclude, we only need to find an "intelligent" algorithm which allows us to bound the revealment of an edge. The first non trivial algorithm that comes to mind goes as follows: begin with the vertex set $V$ containing the origin only and an empty edge set $F$. While there is an edge in $E\setminus F$ which is incident to some vertex in $V$, pick the first (wrt.~to some fixed order) and add it to $F$. If it is open, add the two incident vertices to $V$. Stop as soon as $V\cap \partial\Lambda_n \neq \emptyset$. This algorithm explores the cluster of $0$ starting from the origin. The advantage of this algorithm is that we get a good bound on the revealment of edges that are far away. But edges near the origin still have a revealment close to 1.

The new idea is to average over multiple algorithms. Since the one above seems to work quite well, we will define similar algorithms in the following way. For $k=1,\dots,n$, we define the algorithm $T_k$ which does exactly the same as the one above, but starting with edge set $\partial\Lambda_k$. We add the additional restriction that the algorithm only reveals edges inside the box $\Lambda_n$. Recall that we write $\Lambda_k(v) = v + \Lambda_k$ for the box around $v$. The revealment of an edge $e = \{v,w\}$ is then easy to bound by
\[
\delta_e(T_k) \leq \mathbb{P}_p[v\leftrightarrow\partial\Lambda_k] + \mathbb{P}_p[w\leftrightarrow\partial\Lambda_k] \leq \mathbb{P}_p[v\leftrightarrow \partial\Lambda_{\vert k - \Vert v\Vert_1\vert}(v)] + \mathbb{P}_p[w\leftrightarrow\partial\Lambda_{\vert k - \Vert w\Vert_1\vert}(w)].
\]
Summing gives
\[
\sum_{k=1}^n \delta_e(T_k) \leq 4\sum_{k=0}^{n-1} \theta_k = 4\Sigma_n
\]
by invariance of $\mathbb{P}_p$ under translation. All together:
\begin{align*}
n\theta_n(1-\theta_n) = n\mathrm{Var}_p[1_{0\leftrightarrow\partial\Lambda_n}] &\leq 2\sum_{k=1}^n \sum_{e\in E} \delta_e(T_k)\mathrm{Cov}_p(f, \omega_e) \\
&\leq 8\Sigma_n\sum_{e\in E}\mathrm{Cov}_p(f,\omega_e) \\
&= 8p(1-p)\Sigma_n\theta_n' \leq 2\Sigma_n\theta_n'.
\end{align*}

On every interval of the form $[0, 1-\epsilon]$, we can bound $1-\theta_n$ uniformly in $n$ from below by $c = 1 - \theta_1(1-\epsilon) > 0$. This gives the wanted differential inequality \eqref{eq:diff_n}
\[
\theta_n' \geq \dfrac{c}{2}\cdot\dfrac{n}{\Sigma_n}\theta_n.
\]

\section{Poisson-Boolean Percolation on $\mathbb{R}^d$}\label{ssec:poisson-bool}\label{ss:Poisson-Boolean_Percolation}

This section is based on the article \cite{DRT18}. I will show how to apply the OSSS method to the Poisson-Boolean percolation model. To this end, I will first introduce the model. Then, I will present how Duminil-Copin, Raoufi and Tassion bypass the additional difficulties of the model to apply the OSSS inequality. The major difference to the Bernoulli model is the fact that two regions far away from one another are not independent anymore.

The proofs I present are mainly from the cited paper. I will try to give more insight into the intuitions and the ideas behind the proofs. However, before starting with the model itself, I need to present the notion of Poisson Point Processes. Those accustomed with this object may skip this first part.

\subsection{The Poisson Point Process}

I will derive all necessary properties of Poisson Point Processes we need hereafter. However, I will do so in a more general context. This makes the proofs considerably more readable. On the other hand, it becomes more difficult to relate the abstract objects to the percolation model we are studying afterwards. For this reason, I will first give a vague definition of Poisson Point Processes. I advise to skip the part after the motivation on a first reading. Afterwards, this part will be more easy to understand. Furthermore, this part is not meant to replace textbooks on this subject. In particular, I will not provide proofs for all the facts I use.

For further reading, I refer to \cite{K14} for a general approach and to \cite{MR96} for Poisson Point Processes in the context of Poisson-Boolean Percolation. The following structure heavily relies on both books.\\

The idea behind the Poisson Point Process (PPP) is to plot points randomly and uniformly in the space such that the number of points is also random. What would be the most intuitive way to do so? Let us suppose that we expect some $\lambda$ points in the unit square of $\mathbb{R}^2$. Then dissect your space into squares of volume $\epsilon$. If $\epsilon$ is small enough, each square will contain at most one point with high probability. We can model this situation via independent Bernoulli experiments of parameter $\epsilon\lambda$: In every square, independently, we put a point with probability $\epsilon\lambda$ and no point otherwise. A bounded set $A$ contains approximately $\mathrm{Leb}(A)/\epsilon$ of such squares, where we write $\mathrm{Leb}$ for the Lebesgue measure. Hence, the number of points in $A$ follows a binomial law with parameters $\mathrm{Leb}(A)/\epsilon + o(1)$ and $\epsilon\lambda$, of mean $\lambda\cdot\mathrm{Leb}(A)$. Using the limit theorem for binomial distribution, this would give a Poisson distribution of parameter $\lambda\cdot\mathrm{Leb}(A)$, thence the name. I will not prove this convergence in this part, but I present a proof in the appendix, see Theorem \ref{aprop:approximation}.)

In a more formal way, we would like to have a discrete set of points $\eta\subseteq\mathbb{R}^n$, i.e.~without any limit point, such that
\begin{enumerate}[label=\roman*)]
	\item for every bounded Borel set $A\subseteq\mathbb{R}^n$, one has
	\[
	\vert \eta \cap A\vert \sim \mathrm{Poisson}(\lambda\cdot \mathrm{Leb}(A)),
	\]
	\item for two disjoint bounded Borel sets $A,B\subseteq\mathbb{R}^n$, the number of points $\vert \eta\cap A\vert$ in $A$ and the number of points $\vert \eta\cap B\vert$ in $B$ are independent.
\end{enumerate}
More generally, we may replace the measure $\lambda\cdot \mathrm{Leb}$ by some other Radon measure $\mu$. We say that $\mu$ is the \emph{intensity} of the Poisson Point process.

On a first reading, you may want to jump now to section \ref{sec:boolModel}. I will now present the general construction and properties of Poisson Point Processes.\\

In what follows, let $E$ be a locally compact Polish space, e.g.~$E = \mathbb{R}^d$, with its Borel $\sigma$-algebra $\mathcal{B}(E)$. We will denote by
\[
\mathcal{B}_b(E) := \{ A\in\mathcal{B}(E)\;\vert\; A\text{ is relatively compact}\}
\]
the set of bounded Borel sets. Furthermore, let $\mathcal{M}(E)$ be the set of Radon measures on $(E,\mathcal{B}(E))$. We equip $\mathcal{M}(E)$ with the $\sigma$-algebra
\[
\mathbb{M} := \sigma\{ I_A\;\vert\; A\in \mathcal{B}_b(E)\}
\]
generated by the the functions $I_A:\mathcal{M}(E)\rightarrow\mathbb{R}$ mapping some measure $\mu$ to $\mu(A)$. In other words, $\mathbb{M}$ is the smallest $\sigma$-algebra on $\mathcal{M}(E)$ such that it is allowed (i.e.~measurable) to map a measure to the content it attributes to some Borel set. In many ways, this is a very natural choice, since it is also the Borel $\sigma$-algebra with respect to the vague topology on $\mathcal{M}(E)$.

\begin{defin}
A \emph{random measure} is a random variable defined on a probability space $(\Omega,\mathcal{F},\mathbb{P})$ taking values in $\mathcal{M}(E)$.
\end{defin}

From now on, let $\eta$ be a random measure. From the definition of the $\sigma$-algebra $\mathbb{M}$, it follows that the distribution $P_\eta$ of $\eta$ is entirely determined by the distributions of
\begin{align*}
&\quad\;\{(I_{A_1}\circ \eta,\dots,I_{A_n}\circ\eta)\;\vert\; n\in\mathbb{N},\, A_1,\dots,A_n\in\mathcal{B}_b(E)\} \\
&= \{(\eta(A_1),\dots,\eta(A_n))\;\vert\; n\in\mathbb{N},\, A_1,\dots,A_n\in\mathcal{B}_b(E)\}.
\end{align*}
If for disjoint sets $A_1,\dots,A_n\in\mathcal{B}_b(E)$, the random variables $\eta(A_1),\dots,\eta(A_n)$ are independent, we say that $\eta$ has \emph{independent increments}. In particular, this implies that the distribution of $\eta$ is completely characterised via the distributions of $(\eta(A))_{A\in\mathcal{B}_b(E)}$.

\begin{defin}
Let $\mu \in\mathcal{M}(E)$. We say that $\eta$ is a \emph{Poisson Point Process with intensity $\mu$} (in symbols: $\eta\sim PPP_\mu$) if it has independent increments and if $\eta(A) \sim \mathrm{Poisson}(\mu(A))$ for every $A\in \mathcal{B}_b(E)$. For simplicity, we write that $\eta$ is a $PPP_\mu$.
\end{defin}

\begin{theo}
For every $\mu\in\mathcal{M}(E)$, there exists a $PPP_\mu$.
\end{theo}
\begin{proof}
First, assume $\mu(E)$ to be finite. If $\mu(E) = 0$, the assertion is trivially true. From now on, assume $\mu(E) \in (0,+\infty)$. Let $(X_n)_{n\geq 0}$ be an iid sequence of distribution $\nu := \mu/\mu(E)$. Furthermore, let $N$ be independent of $(X_n)_{n\geq 1}$ of Poisson law of parameter $\mu(E)$. Set
\[
\eta := \sum_{i=1}^N \delta_{X_i}.
\]
One easily verifies that $\eta$ is a random measure. Furthermore, $\eta(A)$ takes only values in $\mathbb{N}$ for every $A\in\mathcal{B}_b(E)$. Let $A_1,\dots,A_l\in\mathcal{B}_b(E)$ be disjoint and consider $k_1,\dots,k_l\in\mathbb{N}$. On the event $\{N = n\}$, we have $\eta(A_i) = k_i$ for every $1\leq i\leq l$ if and only if there are disjoint sets of indices $I_1,\dots,I_l\subseteq\{1,\dots,n\}$ such that $\vert I_i\vert = k_i$ and $X_j \in A_i$ if and only if $j\in I_i$ for every $1\leq i\leq l$. Hence, if $\mathcal{I}$ is the set of all sequences of all such index sets, then
\begin{align*}
&\quad\;\mathbb{P}[\eta(A_1) = k_1,\dots,\eta(A_n) = k_n] \\
&= \sum_{n\in\mathbb{N}} \mathbb{P}[\eta(A_1) = k_1,\dots,\eta(A_n) = k_n\;\vert\; N = n]\cdot\mathbb{P}[N = n]\\
&= \sum_{n\geq \sum_{i=1}^l k_i} \sum_{(I_1,\dots,I_l)\in\mathcal{I}} \mathbb{P}\left[\left(\bigcap_{i=1}^l \bigcap_{j\in I_i} \{X_j\in A_i\}\right) \cap \bigcap_{j\in \{1,\dots,n\}\setminus\bigcup_{i=1}^l I_i} \llbrace X_j\not\in \bigcup_{i=1}^l A_i\rrbrace\right]\cdot\mathbb{P}[N = n].
\end{align*}
Set $k := \sum_{i=1}^l k_i$. By independence, we may transform the intersections into products to obtain
\begin{align*}
&= \sum_{n\geq \sum_{i=1}^l k_i} \sum_{(I_1,\dots,I_l)\in\mathcal{I}} \left(\prod_{i=1}^l\prod_{j\in I_i} \nu(A_i)\right)\cdot\left( \prod_{j\in \{1,\dots,n\}\setminus\bigcup_{i=1}^l I_i} \left(1 - \sum_{i=1}^l \nu(A_i)\right)\right)\cdot\mathbb{P}[N = n]\\
&= \sum_{n\geq k} \dfrac{n!}{k_1!k_2!\dots k_l!(n-k)!} \left(\prod_{i=1}^l \nu(A_i)^{k_i}\right)\cdot \left(1 - \sum_{i=1}^l \nu(A_i)\right)^{n - k}\cdot \dfrac{\mu(E)^n}{n!}e^{-\mu(E)}\\
&= \sum_{n\geq k} \dfrac{n!}{k_1!k_2!\dots k_l!(n-k)!} \left(\prod_{i=1}^l \mu(A_i)^{k_i}\right)\cdot \left(\mu(E) - \sum_{i=1}^l \mu(A_i)\right)^{n - k}\cdot \dfrac{1}{n!}e^{-\mu(E)}.
\end{align*}
The $n!$ simplifies and by changing the indices, we get
\begin{align*}
&= \sum_{n\geq 0} \dfrac{1}{k_1!k_2!\dots k_l!n!} \left(\prod_{i=1}^l \mu(A_i)^{k_i}\right)\cdot \left(\mu(E) - \sum_{i=1}^l \mu(A_i)\right)^{n}\cdot e^{-\mu(E)}\qquad\qquad\qquad\qquad\\
&= \left(\prod_{i=1}^l \dfrac{\mu(A_i)^{k_i}}{k_i}\right)\cdot \sum_{n\geq 0} \dfrac{1}{n!}\left(\mu(E) - \sum_{i=1}^k \mu(A_i)\right)\cdot e^{-\mu(E)}\\
&= \left(\prod_{i=1}^l \dfrac{\mu(A_i)^{k_i}}{k_i}\right)\cdot e^{\mu(E) - \sum_{i=1}^l \mu(A_i)}\cdot e^{-\mu(E)} = \left(\prod_{i=1}^l \dfrac{\mu(A_i)^{k_i}}{k_i}e^{-\mu(A_i)}\right).
\end{align*}
We conclude that the random variables $\eta(A_1),\dots,\eta(A_l)$ are independent of Poisson law with parameters $\mu(A_1),\dots,\mu(A_l)$. We conclude that $\eta$ is a $PPP_\mu$.

Let us return to a general $\mu\in\mathcal{M}(E)$. Then $\mu$ is $\sigma$-finite. Hence, there exists a sequence $(E_n)_{n\in\mathbb{N}}\subset\mathcal{B}(E)$ such that $E_n\uparrow E$ and $\mu(E_n) < \infty$ for every $n$. Define the finite measures $\mu_{n} := \mu\big{(} (E_{n}\setminus E_{n-1})\cap \cdot\big{)}$ for every $n\in\mathbb{N}$, where $E_{-1} := \emptyset$. Consider independent Poisson Point Processes $(\eta_n)_{n\in\mathbb{N}}$ with intensities $(\mu_n)_{n\in\mathbb{N}}$ respectively. Finally, define
\[
\eta := \sum_{n\in\mathbb{N}} \eta_n.
\]
If $A\in\mathcal{B}(E)$, then there exists $n\in\mathbb{N}$ such that $A\subseteq E_n$. This means
\begin{align*}
P_{\eta(A)}  &= P_{\eta_0(A) + \dots + \eta_n(A)}\\
&=P_{\eta_0(A)} \ast \dots\ast P_{\eta_n(A)} \\
&= \mathrm{Poisson}(\mu_0(A))\ast\dots\ast\mathrm{Poisson}(\mu_n(A)) \\
&= \mathrm{Poisson}\left(\sum_{i=0}^n \mu_i(A)\right) = \mathrm{Poisson}(\mu(A)).
\end{align*}
The independence of $\eta(A_1),\dots,\eta(A_l)$ for disjoint bounded Borel sets $A_1,\dots,A_l$ is shown in the same way: take $n$ large enough such that $E_n$ contains all the sets and decompose them into $E_0\cap A_i$ and $(E_{m+1}\setminus E_m)\cap A_i$, $0\leq m\leq n-1$. 
\end{proof}

From now on, let $\eta$ be a $PPP_\mu$ as above. A very important consequence of the construction is the following: If $\mu$ has no atoms, then $\eta(\{x\}) \leq 1$ for all $x\in E$ a.s. In other words, we may identify $\eta$ with the (random) set $\{x\in E\;\vert\; \eta(\{x\}) = 1\}$. If $\mu$ has atoms, $\eta$ can still be seen as a multiset. From now on, we will make no distinction between the random measure $\eta$ and the corresponding random (multi)subset of $E$. Note that $\vert \eta\cap A\vert$ is a.s.~finite for every bounded Borel set of $E$. In particular, $\eta$ has no limit points. Furthermore, since $\mu$ is $\sigma$-finite, $\eta$ contains at most countably many points. The set we described in the motivation is obtained with $\mu = \lambda\cdot\mathrm{Leb}$. Before having a closer look at the Poisson point process we will need later on, we will derive some very useful properties of Poisson point processes.

\begin{prop}
Let $\eta_1$ and $\eta_2$ be two independent Poisson point processes with intensities $\mu_1$ and $\mu_2$ respectively. Then $\eta := \eta_1 + \eta_2$ is a Poisson point process with intensity $\mu_1 + \mu_2$.
\end{prop}
\begin{proof}
Let $A_1,\dots,A_l\in\mathcal{B}_b(E)$ be disjoint and $k_1,\dots,k_l\in\mathbb{N}$. Then
\begin{align*}
&\;\mathbb{P}\left[\eta(A_1) = k_1,\dots, \eta(A_l) = k_l\right]\\
&= \sum_{i_1 = 0}^{k_1}\cdots\sum_{i_l=0}^{k_l} \mathbb{P}\left[\eta_1(A_1) = i_1, \eta_2(A_1) = k_1 - i_1, \dots, \eta_1(A_l) = i_l, \eta_2(A_l) = k_l - i_l\right].
\end{align*}
Using the fact that $\eta_1$ and $\eta_2$ are independent Poisson point processes, we get
\begin{align*}
&= \sum_{i_1 = 0}^{k_1}\cdots\sum_{i_l=0}^{k_l} \prod_{j=1}^l \dfrac{\mu_1(A_j)^{i_j}}{i_j!}e^{-\mu_1(A_j)}\cdot\dfrac{\mu_2(A_j)^{k_j-i_j}}{(k_j-i_j)!}e^{-\mu_2(A_j)}\\
&= \left(\prod_{j=1}^l \dfrac{\left(\mu_1(A_j) + \mu_2(A_j)\right)^k_j}{k_j!}e^{-\mu_1(A_j) - \mu_2(A_j)}\right)\cdot  \underbrace{\sum_{i_1 = 0}^{k_1}\cdots\sum_{i_l=0}^{k_l}\prod_{j=1}^l \binom{k_j}{i_j}\dfrac{\mu_1(A_j)^{i_j}\mu_2(A_j)^{k_j-i_j}}{(\mu_1(A_j) + \mu_2(A_j))^{k_j}}}_{=1},
\end{align*}
where we use the Binomial Formula and the fact that we may rewrite the last sum as 
\[
\prod_{j=1}^l \left(\sum_{i = 0}^{k_j} \binom{k_j}{i} \dfrac{\mu_1(A_j)^{i_j}\mu_2(A_j)^{k_j-i_j}}{(\mu_1(A_j) + \mu_2(A_j))^{k_j}}\right).
\]
The result follows as usual.
\end{proof}

The last result we will prove before concentrating on the Poisson-Boolean case is known under different names. First proved by Mecke, it is often called Mecke's eqution or Mecke's formula.

\begin{theo}[Mecke's formula, \cite{M67}]\label{theo:mecke}
Let $\Phi: E\times \mathcal{M}(E)\rightarrow [0,+\infty]$ be positive measurable. Then
\[
\mathbb{E}\left[\int_E \Phi(x, \eta) \d{\eta(x)}\right] = \int_E\mathbb{E}\left[ \Phi(x,\eta + \delta_x) \right]\d{\mu(x)}.
\]
\end{theo}
\begin{proof}
As before, assume first that $\mu$ is finite. Then the above construction yields
\begin{align*}
\mathbb{E}\left[\int_E \Phi(x, \eta) \d{\eta(x)}\right] &= \mathbb{E}\left[\sum_{i=1}^N \Phi(X_i,\eta)\right]\\
&= \mathbb{E}\left[\mathbb{E}\left[\left.\sum_{i=1}^N \Phi\left(X_i, \sum_{j=1}^N \delta_{X_j}\right)\;\right\vert\; N\right]\right]\\
&= \sum_{n\geq 1} \sum_{i=1}^n \mathbb{E}\left[\Phi\left(X_i, \sum_{j=1}^n \delta_{X_j}\right)\right]\cdot \mathbb{P}(N = n)\\
&= \sum_{n\geq 1}\sum_{i=1}^n \mathbb{E}\left[\mathbb{E}\left[\left.\Phi\left(X_i, \sum_{j=1}^n \delta_{X_j}\right)\;\right\vert\; X_i\right]\right]\cdot \dfrac{\mu(E)^n}{n!}e^{-\mu(E)}\\
&= \sum_{n\geq 1}\sum_{i=1}^n \int_E \mathbb{E}\left[\Phi\left(x, \delta_x + \sum_{j=1, j\neq i}^n \delta_{X_j}\right)\right]\d{\dfrac{\mu(x)}{\mu(E)}}\cdot \dfrac{\mu(E)^n}{n!}e^{-\mu(E)}.
\end{align*}
Since the sequence $(X_n)_{n\geq 1}$ is iid, we may replace $\sum_{i\neq j}$ by $\sum_{i=1}^{n-1}$, giving
\begin{align*}
\qquad\;\qquad\qquad\qquad&= \int_E \sum_{n\geq 1}\sum_{i=1}^n \mathbb{E}\left[\Phi\left(x, \delta_x + \sum_{j=1}^{n-1}\delta_{X_j}\right)\right] \cdot \dfrac{\mu(E)^{n-1}}{n!}e^{-\mu(E)}\d{\mu(x)}.
\end{align*}
The random variable is now independent of $i$, hence
\begin{align*}
\qquad\;\qquad\qquad\quad&= \int_E \sum_{n\geq 1} n\cdot\mathbb{E}\left[\Phi\left(x, \delta_x + \sum_{j=1}^{n-1}\delta_{X_j}\right)\right] \cdot \dfrac{\mu(E)^{n-1}}{n!}e^{-\mu(E)}\d{\mu(x)}\\
&= \int_E \sum_{n\geq 1}\mathbb{E}\left[ \Phi\left(x, \delta_x + \sum_{j=1}^{n-1}\delta_{X_j}\right)\right] \cdot \dfrac{\mu(E)^{n-1}}{(n-1)!}e^{-\mu(E)}\d{\mu(x)}\\
&= \int_E \sum_{n\geq 0} \mathbb{E}\left[\Phi\left(x, \delta_x + \sum_{j=1}^{n}\delta_{X_j}\right)\right]\cdot \dfrac{\mu(E)^{n}}{n!}e^{-\mu(E)} \d{\mu(x)}.
\end{align*} 
The result follows via
\[
\sum_{n\geq 0} \mathbb{E}\left[\Phi\left(x, \delta_x + \sum_{j=1}^{n}\delta_{X_j}\right)\right] \cdot \dfrac{\mu(E)^{n}}{n!}e^{-\mu(E)} = \mathbb{E}\left[\Phi\left(x, \delta_x + \sum_{j=1}^N \delta_{X_j}\right)\right].
\]

The generalisation to a $\sigma$-finite measure is straight forward, but very technical and tedious. I therefore present only the outline of it. Instead of proving Mecke's formula for all functions, it suffices to show it for indicator functions of sets $A\times B$ with $A\in\mathcal{B}(E)$ and $B\in\mathbb{M}$. To be precise, we will take even less sets, making sure that they still generate the entire $\sigma$-algebra. Indeed, we will take the sequence $(E_n)$ from before and we will look at the restriction $\eta_n$ to $E_n$, considering only sets, where the formula works after replacing $\eta$ with $\eta_n$. The rest of the proof can be found in Mecke's original paper \cite{M67}.
\end{proof}

Note that
\[
\int_E \Phi(x, \eta)\d{\eta(x)} = \sum_{x\in\eta} \Phi(x,\eta),
\]
where we see $\eta$ as a random (multi)set.\\

Let us now have a look at the Poisson point process which we will use for the Poisson-Boolean model. For this, we take $E = \mathbb{R}^d\times\mathbb{R}_+$ and $\mu = \lambda\mathrm{d}z\otimes \nu$, where $\mathrm{d}z$ denotes the Lebesgue measure on $\mathbb{R}^d$ and where $\nu$ is some probability measure on $\mathbb{R}_+$. For the intuition, we interpret a point $(x,r)\in\eta$ as a ball with center $x$ and radius $r$. For this intuition to be useful, we would like to have the representation
\[
\eta = \{(x,r_x)\;\vert\; x\in \xi\},
\]
where $\xi$ is a Poisson point process with intensity $\lambda\mathrm{d}z$ and $(r_x)_{x\in\mathbb{R}^d}$ is an iid family of random variables of law $\nu$. Indeed, we have this equality in law: take two bounded Borel sets $A\subseteq\mathbb{R}^d$ and $B\subseteq\mathbb{R}_+$ and $k\in\mathbb{N}$. Then
\begin{align*}
\mathbb{P}[\vert \{(x,r_x)\;\vert\; x\in\xi\}\cap A\times B\vert = k] 
&= \mathbb{P}\left[\vert \{r_x\;\vert\; x\in\xi\cap A\}\cap B\vert = k\right]\\
&= \sum_{i\geq k} \mathbb{P}\left[ \vert \{r_1,\dots,r_i\}\cap B\vert = k\right]\cdot\mathbb{P}[\vert \xi\cap A\vert = i]\\
&= \sum_{i\geq k} \binom{i}{k} \left(\nu(B)^k\cdot (1-\nu(B))^{i-k}\right)\cdot \dfrac{\big{(}\lambda\mathrm{Leb}(A)\big{)}^i}{i!}e^{-\lambda\mathrm{Leb}(A)}\\
&= \dfrac{\big{(}\nu(B)\lambda\mathrm{Leb}(A)\big{)}^k}{k!}e^{-\lambda\mathrm{Leb}(A)} \sum_{i\geq 0} \dfrac{(1 - \nu(B))^i\cdot (\lambda\mathrm{Leb}(A))^i}{i!}\\
&=  \dfrac{\big{(}\nu(B)\lambda\mathrm{Leb}(A)\big{)}^k}{k!}e^{-\lambda\mathrm{Leb}(A)}\cdot e^{(1 - \nu(B))\lambda\mathrm{Leb}(A)}\\
&=  \dfrac{\big{(}\nu(B)\lambda\mathrm{Leb}(A)\big{)}^k}{k!}e^{-\lambda\nu(B)\mathrm{Leb}(A)} = \mathbb{P}[\eta(A\times B) = k],
\end{align*}
where $(r_i)_{i\in\mathbb{N}}$ is an iid family of random variables of law $\nu$, independent of all other variables. Since $\{A\times B\;\vert\; A\in\mathcal{B}(\mathbb{R}^d), B\in\mathcal{B}(\mathbb{R}_+)\text{ bounded}\}$ is a generator of the Borel sets of $\mathbb{R}^d\times\mathbb{R}_+$ which is stable under finite intersection, the claim follows.

Later on, we will need some more properties of this specific PPP. Since they need a more extensive vocabulary, I will not present their proofs here. Furthermore, the proofs can be very tedious, which is why I will only present a part of them which can be found in the appendix.

\subsection{The Model: Definitions, Notations and Basic Properties}\label{sec:boolModel}

Let $\lambda > 0$ and $\nu$ a probability measure on $\mathbb{R}_+^*$. We consider a Poisson point process (PPP) $\eta$ on $\mathbb{R}^d\times\mathbb{R}_+$ with intensity $\lambda\mathrm{d}z\otimes \nu$, where $\mathrm{d}z$ denotes the Lebesgue measure on $\mathbb{R}^d$. We will write $\mathbb{P}_\lambda$ for the corresponding probability measure, assuming $\nu$ to be fixed. In the following, we will exclusively work with the representation
\[
\eta = \{ (x, r_x)\;\vert\; x\in\xi\}
\]
for some homogeneous PPP $\xi$ on $\mathbb{R}^d$. Furthermore, we will work on the event with probability one on which every compact set of $\mathbb{R}^d\times\mathbb{R}_+$ contains a finite number of points only. We then define the \emph{occupied set} as
\[
\mathcal{O}(\eta) := \bigcup_{(x,r_x)\in\eta} B(x, r_x),
\]
where $B_r^x := B(x,r)$ is the closed ball of centre $x\in\mathbb{R}^d$ and radius $r > 0$. 

\begin{figure}[ht]  
\centering
  \begin{subfigure}[b]{0.33\linewidth}
    \centering
    \includegraphics[width=.8\linewidth]{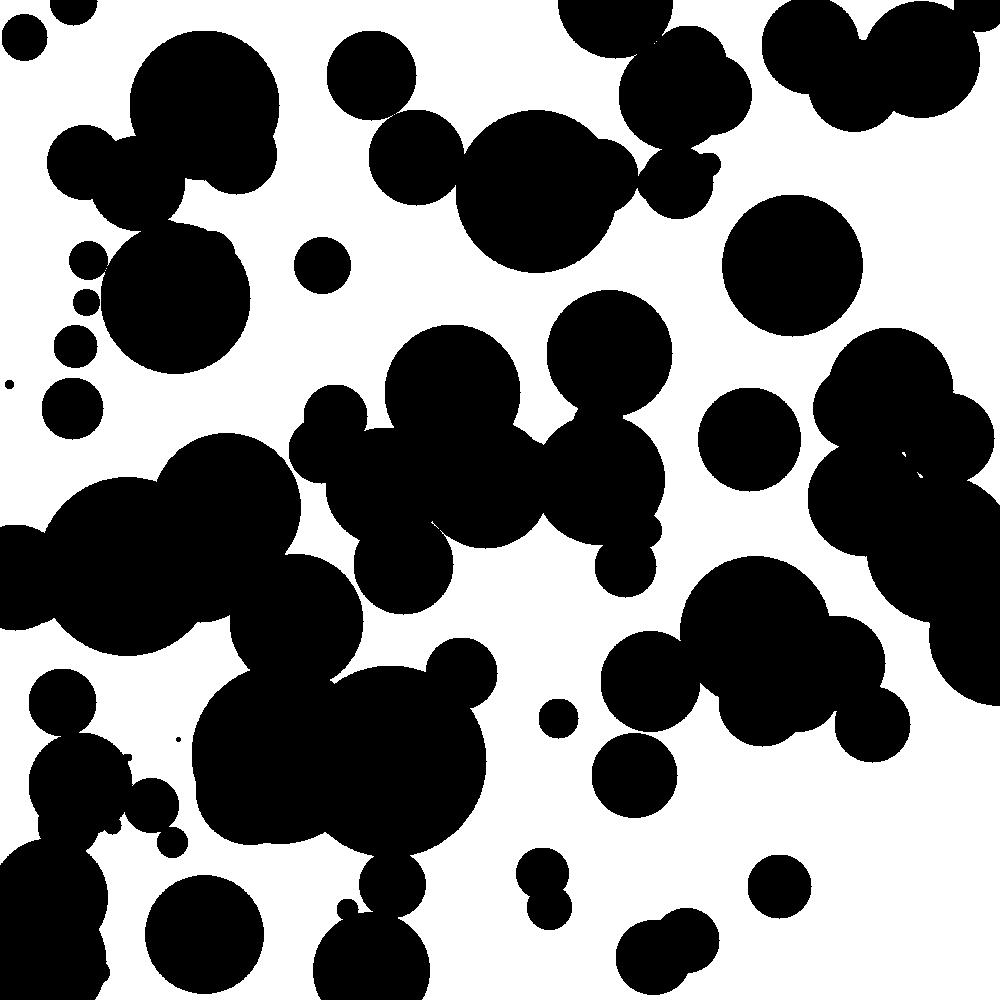}
  \end{subfigure}
  \begin{subfigure}[b]{0.33\linewidth}
    \centering
    \includegraphics[width=.8\linewidth]{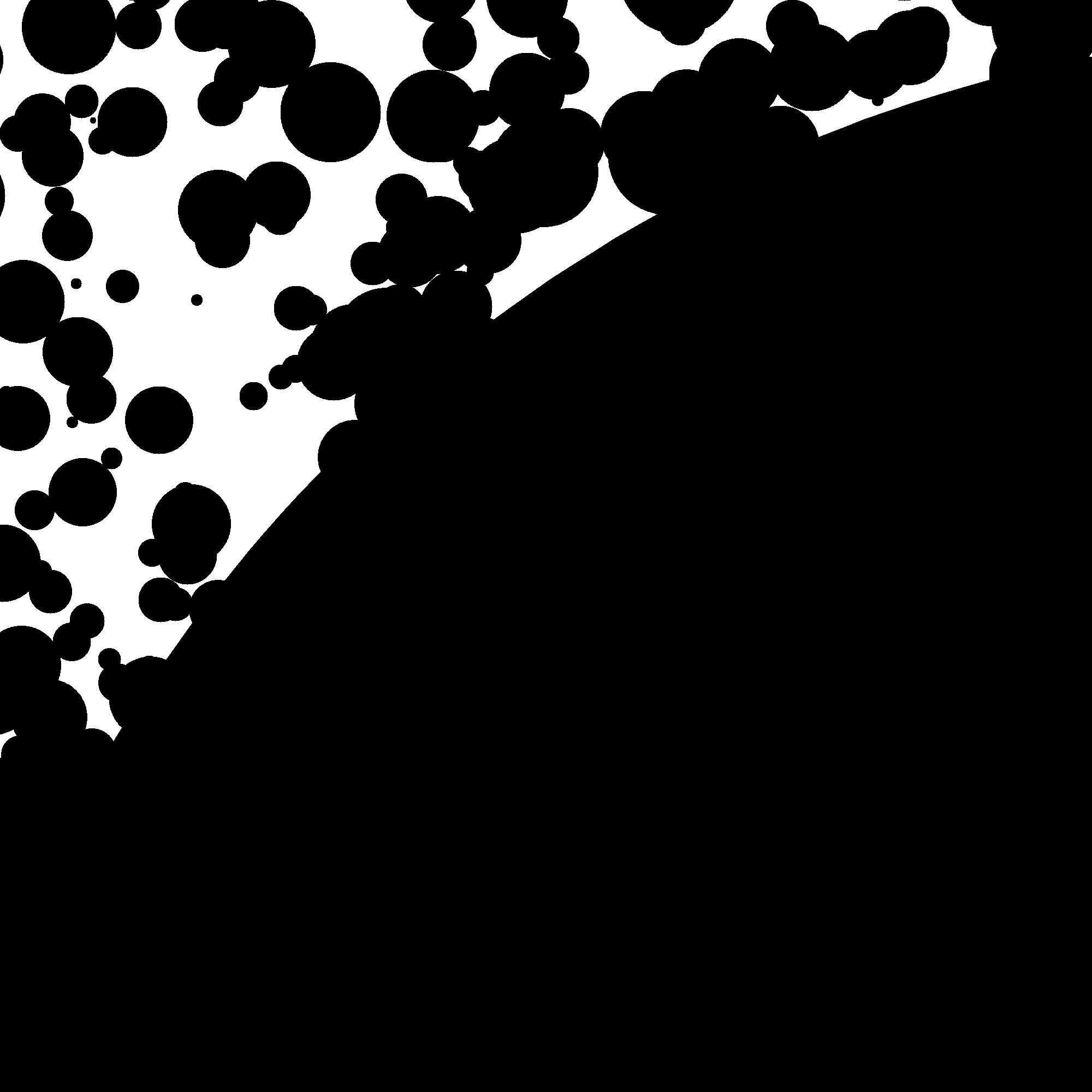}
  \end{subfigure} 
  \begin{subfigure}[b]{0.33\linewidth}
    \centering
    \includegraphics[width=.8\linewidth]{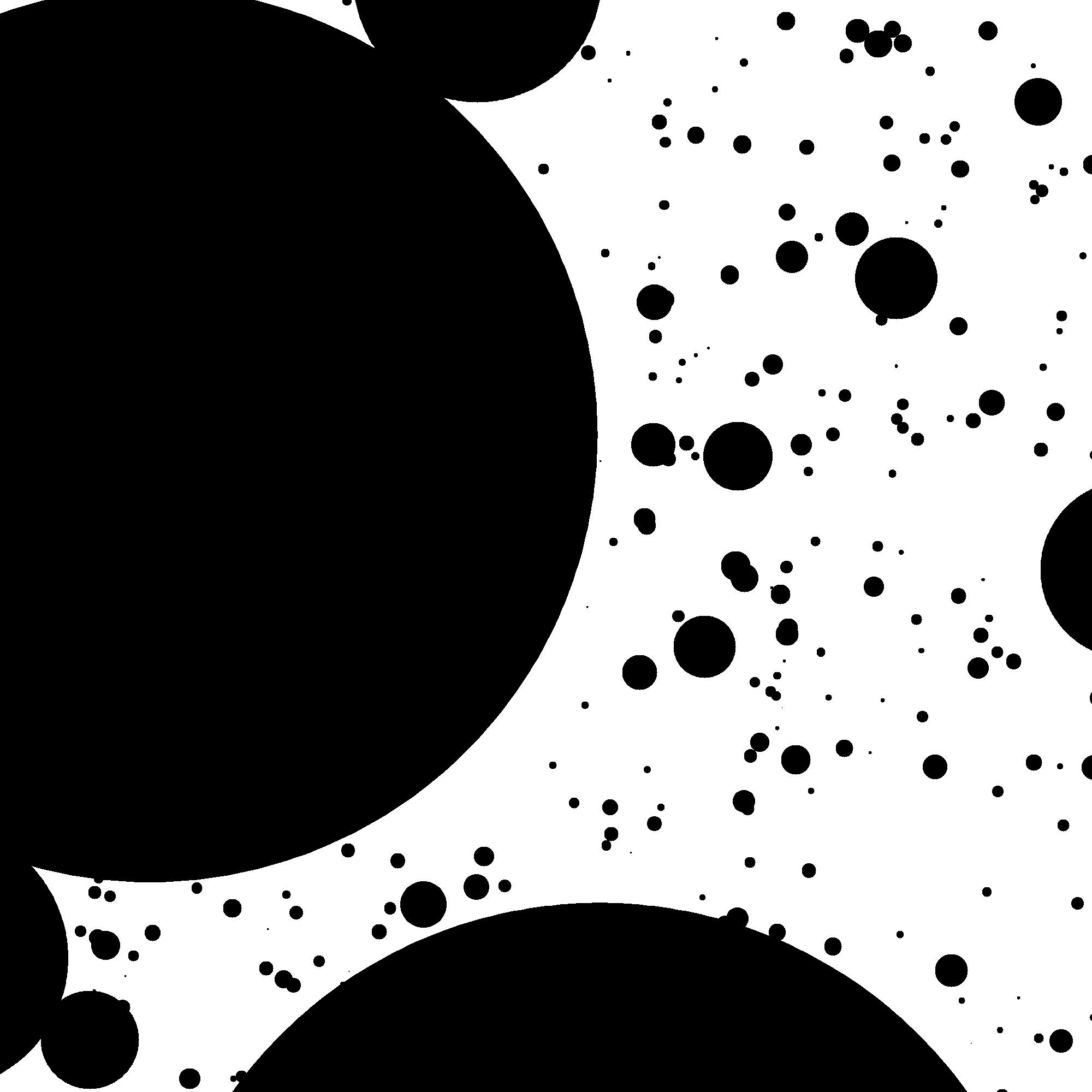}
  \end{subfigure}
  \caption{Samples of Poisson-Boolean models for different intensities and radii distributions} 
\end{figure}

For $x = 0$, we simplify the notation by writing $B_r := B_r^0$. Furthermore, we will denote by $S_r := \partial B_r$ the sphere of radius $r$. For $x,y\in\mathbb{R}^d$ and $A\subseteq\mathbb{R}^d$, we write $x\leftrightarrow y$, if there exists a continuous path from $x$ to $y$ which lies entirely in $\mathcal{O}(\eta)$, and we write $x\leftrightarrow A$, if $x\leftrightarrow y$ for some $y\in A$. (We use the convention that $x\leftrightarrow x$ for all $x\in\mathbb{R}^d$.) Finally, we write $x\leftrightarrow \infty$, if $x$ belongs to an unbounded connected component of $\mathcal{O}(\eta)$. If we consider only the balls which are \emph{entirely contained} in some set $Z\subseteq\mathbb{R}^d$, we replace $\leftrightarrow$ by $\overset{Z}{\leftrightarrow}$. Note that this is not equivalent to the fact that two points are connected in $\mathcal{O}(\eta)\cap Z$. We will write $\eta(Z)$ for the induced process (cf. Figure \ref{fig:eta(Z)}). 
\begin{figure}[ht]
\centering
  \begin{subfigure}[b]{0.49\linewidth}
    \centering
    \includegraphics[width=.8\linewidth]{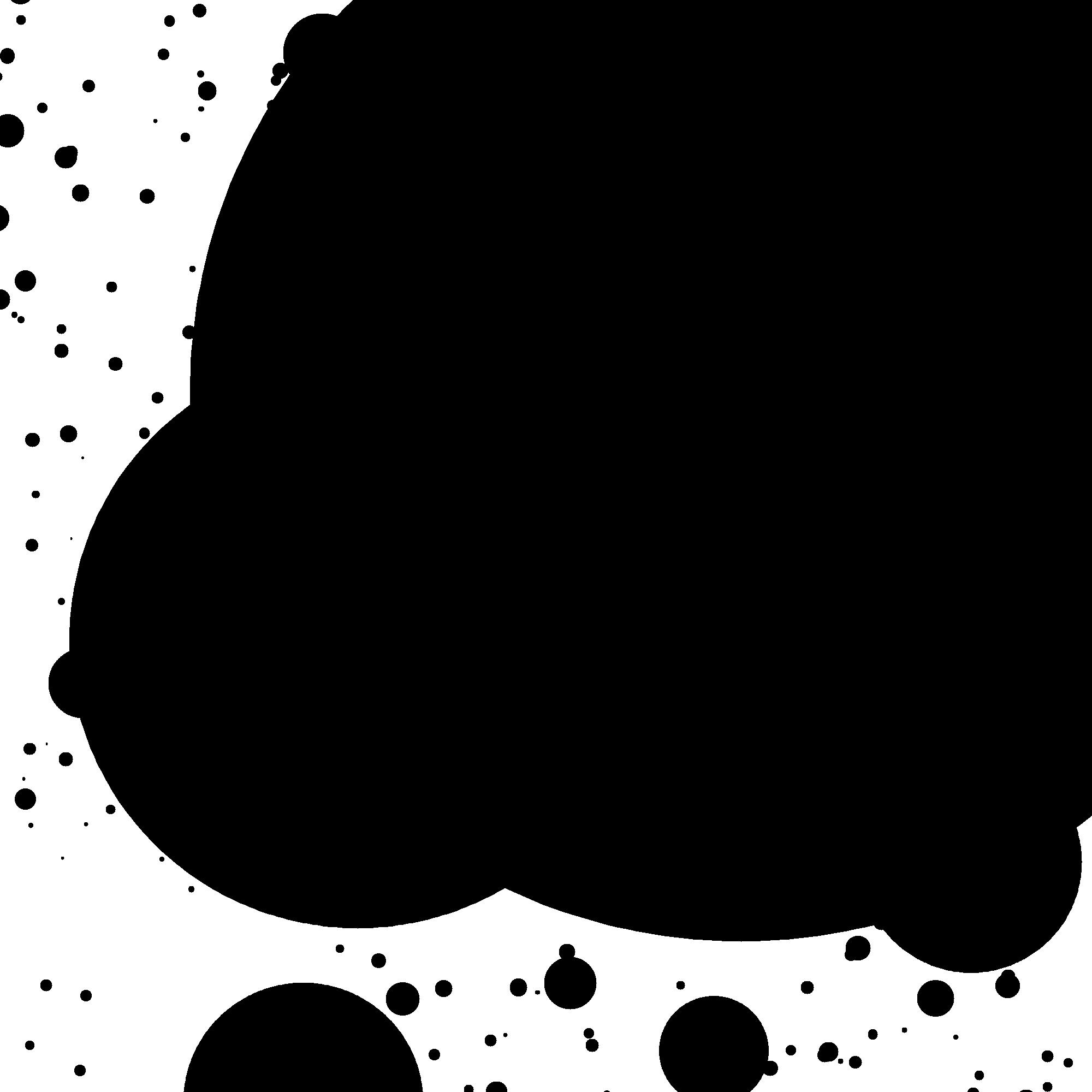}
    \caption{$\mathcal{O}(\eta) \cap Z$}
  \end{subfigure}
  \begin{subfigure}[b]{0.49\linewidth}
    \centering
    \includegraphics[width=.8\linewidth]{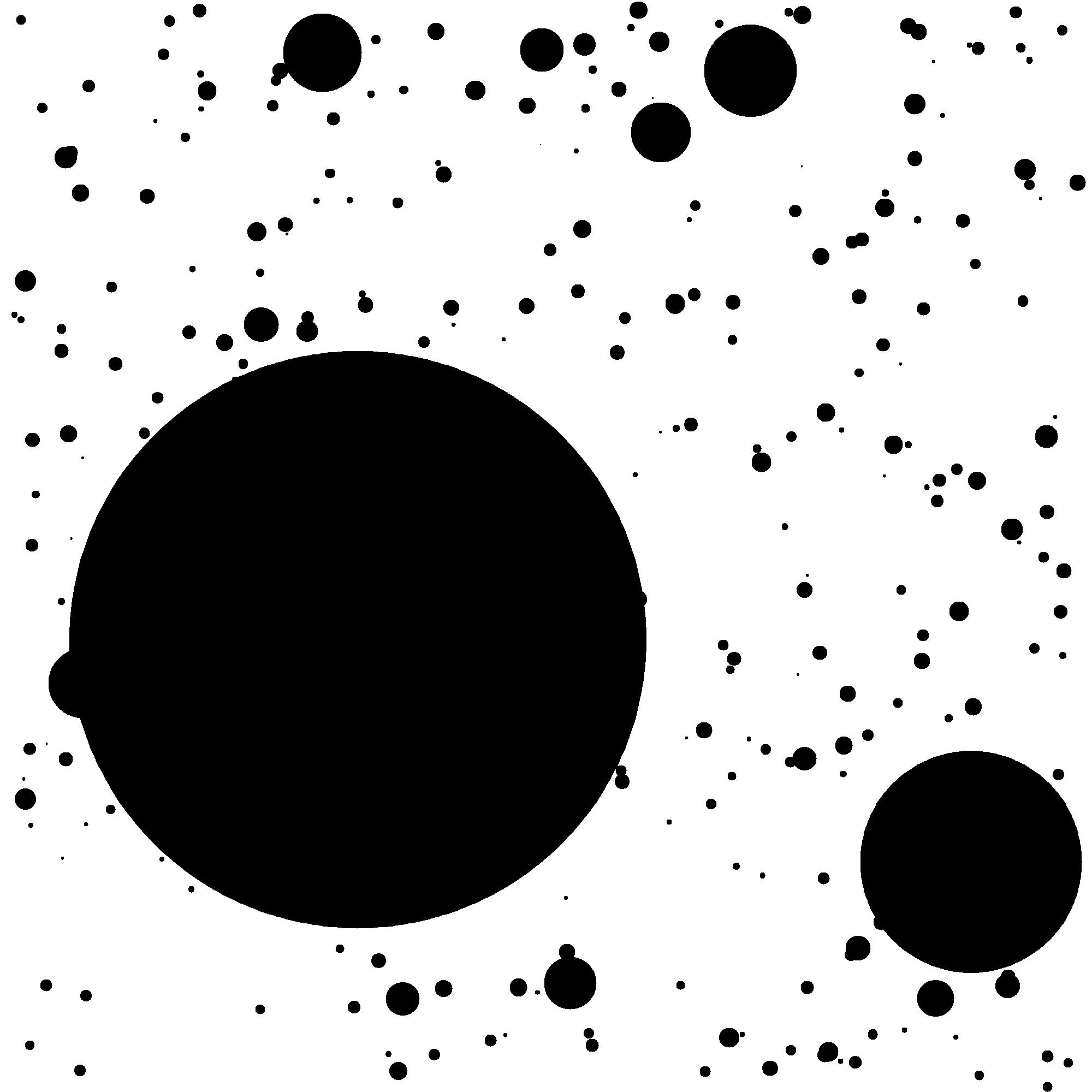}
    \caption{$\mathcal{O}(\eta(Z))$}
  \end{subfigure}
  \caption{Difference between $\eta$ and $\eta(Z)$, where $Z$ is the visible box}
  \label{fig:eta(Z)} 
\end{figure}

As for the Bernoulli percolation model, all the interesting sets are measurable, thus events. The proof of this fact is very unpleasant, so we will just have a look at the general idea. This is particularly interesting, because it uses the method of discretization which we will need later on. A complete proof of the measurability of $\{x\leftrightarrow y\}$ can be found in the appendix.

For $\epsilon > 0$, consider the boxes $\eta_{x,r} = \eta \cap [x, x+\epsilon)^d\times [r, r+\epsilon)$ with $(x,r)\in\epsilon (\mathbb{Z}^d\times\mathbb{N})$. Now, if a box $\eta_{x,r}$ contains one or more points, we replace them by the point $(x,r)$. From now on, we will work on the event of probability 1 that no two balls from the model are tangent. In particular, if we consider only a finite number of occupied balls, then they intersect if and only if the approximations intersect for $\epsilon$ small enough. Note, that in the approximation, we only consider a countable number of points, thus measurability is not a big problem. Finally, we use the fact that every ball $B_n^x$ contains only a finite number of occupied balls and the fact that
\[
\{x \leftrightarrow y\} = \bigcup_{n\in\mathbb{N}} \{ x\overset{B_n^x}{\leftrightarrow} y\}.
\]

Before continuing, we first need to discuss, if the model is interesting, i.e. we need to make sure that $\mathcal{O}(\eta)\neq \mathbb{R}^d$.
\begin{prop}
The following two assertions are equivalent.
\begin{enumerate}[label=\roman*)]
	\item The law $\nu$ has a finite $d$-th moment.
	\item The Poisson-Boolean model is non trivial, i.e. $\mathcal{O}(\eta)\neq \mathbb{R}^d$ almost surely.
\end{enumerate}
\end{prop}
\begin{proof}
See Proposition \ref{aprop:pb_non_trivial} in the appendix.
\end{proof}

Henceforth, we will assume that $\nu$ has a finite $d$-th moment. As for the discrete model, we introduce the quantities
\[
\theta_r(\lambda) := \mathbb{P}_\lambda[0 \leftrightarrow S_r]\quad\text{ and }\quad \theta(\lambda) := \mathbb{P}_\lambda[0\leftrightarrow \infty] = \lim_{r\to+\infty} \theta_r(\lambda)
\]
where $r\in\mathbb{R}_+$. The critical point of percolation is then given by
\[
\lambda_c := \inf \{\lambda > 0\;\vert\; \theta(\lambda) > 0\}.
\]
Unfortunately, it will not be possible in general to show exponential decay of $\theta_r$ in the subcritical regime $\lambda < \lambda_c$. Moreover, the related quantity
\[
\widehat{\lambda}_c := \inf\{\lambda > 0\;\vert\; \liminf_{r \to +\infty} \mathbb{P}_\lambda[S_r\leftrightarrow S_{2r}] > 0\}
\]
may differ from $\lambda_c$. In general, we only have the inequality
\[
\widehat{\lambda}_c \leq \lambda_c.
\]
The OSSS method applied to the Poisson-Boolean model will eventually prove that $\widehat{\lambda}_c = \lambda_c$ for a large class of radii distributions $\nu$.

As before, we need to show that the model is not trivial. This means $\lambda_c,\widehat{\lambda}_c \in (0,+\infty)$. As we will need it later on, we will directly add a more subtle statement.

\begin{theo}\label{theo:open_set}
If $d\geq 2$, then $0 < \widehat{\lambda}_c \leq {\lambda}_c < +\infty$. Furthermore, the set
\[
\{ \lambda > 0\;\vert\; \liminf_{r \to +\infty} \mathbb{P}_\lambda[S_r\leftrightarrow S_{2r}] = 0\}
\]
is non empty and open.
\end{theo}

The proof is non trivial and we will only prove a small part. The complete proof is very technical and does not give further insights. Thus, I refer to \cite[Appendix A]{GT17} for the missing proof. For the rest of Section \ref{ssec:poisson-bool}, we will assume $d \geq 2$.

\begin{lem}
It holds ${\lambda}_c < +\infty$.
\end{lem}
\begin{proof}
Let $\epsilon > 0$ such that $c := \nu([\epsilon,+\infty)) > 0$. Then, consider the point measure $\nu' := c\delta_\epsilon$ and the induced Poisson point process $\eta'$ on $\mathbb{R}^d\times\mathbb{R}_+$ with intensity $\lambda\mathrm{d}z\otimes\nu'$. In other words, we shrink all balls with radius greater or equal $\epsilon$ to balls with radius $\epsilon$ and we delete all balls with smaller radii. Let $P_\lambda$ denote the associated probability measure. Clearly,
\[
\mathbb{P}_\lambda [S_r \leftrightarrow S_{2r}] \geq P_\lambda(S_r \leftrightarrow S_{2r}).
\]
Hence, it suffices to show that the critical intensity ${\lambda}_c(\nu')$ is finite. Now, we consider the net $\epsilon\mathbb{Z}^d$. We say that the site $x\in\epsilon\mathbb{Z}^d$ is occupied if the square $S_x := x + [-\epsilon/2,\epsilon/2)^d$ contains the center of a ball from the modified Poisson point process $\eta'$. This gives a site percolation model\footnote{For a discussion of the difference between site and bond percolation, refer to \cite[Section 1.6]{G99}.} with parameter 
\[
p := P(\vert \eta'\cap S_x\vert \neq 0) = 1 - \exp\left(-\lambda c\epsilon^d\right).
\]
As for the bond percolation model, the critical parameter of the site percolation model is strictly inferior to 1. Hence, by choosing $\lambda$ sufficiently large, we obtain percolation in the above model. But if the site percolation model percolates, then
\[
P(0\leftrightarrow \infty) > 0.
\]
\end{proof}

Now that we have the basic properties of our model, we need to develop the two basic techniques in percolation theory.

\subsection{The FKG inequality and Russo's formula}

Since the ideas and the formulations are quite similar to the discrete model, this section is very short and contains only the formulations and the proofs of the FKG inequality and Russo's formula. Note that more general formulations exist.

\begin{defin}
We say that an event $A$ is \emph{increasing}, if its probability increases when we add balls to the Poisson point process. More formally: if $\eta\in A$ and $\eta\subseteq\eta'$, then $\eta'\in A$.
\end{defin}

The idea of the proof of the FKG inequality in the Poisson-Boolean case is to use a discrete version of the inequality. Unfortunately, the version we used for the discrete bond percolation does not suffice in this case. Since the proof of the following version does not differ much from the first one, we will omit it.

\begin{lem}[Generalised discrete FKG inequality]
Consider $p_0,p_1,\ldots \in [0,1]$ such that $\sum_j p_j = 1$. Now, let $\Omega := \mathbb{N}^{\mathbb{Z}^d}$ endowed with the product measure $P$ satisfying $P(\text{site $x$ has value $j$}) = p_j$ for all $x\in\mathbb{Z}^d$ and all $j\in\mathbb{N}$. If $X,Y:\Omega\rightarrow\mathbb{R}$ are two bounded and increasing random variables, then
\[
E(XY) \geq E(X)\cdot E(Y),
\]
where $E$ denotes the expectation with respect to $P$.
\end{lem}

Let us prove now the FKG inequality in the Poisson-Boolean setting.

\begin{theo}[FKG inequality]
Let $A$ and $B$ be two increasing events. Then
\[
\mathbb{P}_\lambda[A\cap B] \geq \mathbb{P}_\lambda[A]\cdot\mathbb{P}_\lambda[B]
\]
for all $\lambda > 0$.
\end{theo}
\begin{proof}
The proof is constructed on a discretisation argument. First, fix $\epsilon > 0$ and consider the lattice $\mathbb{L} := \epsilon(\mathbb{Z}^d\times\mathbb{N})$. For every site $x\in\mathbb{L}$, we consider the induced square $S_x := x + [0,\epsilon)^d$. Then, the random variables $N_x^\epsilon := \vert \eta \cap S_x\vert$ for $x\in\mathbb{L}$ are independent. Denote by $\mathcal{F}_\epsilon$ the $\sigma$-algebra induced by the family $(N_x^\epsilon)_{x\in\mathbb{L}}$. Now, let $f_\epsilon := \mathbb{E}_\lambda [1_A\;\vert\; \mathcal{F}_\epsilon]$ and $g_\epsilon := \mathbb{E}_\lambda [1_B\;\vert\; \mathcal{F}_\epsilon]$. These are increasing functions depending only on $(N_x^\epsilon)_x$, hence
\[
\mathbb{E}_\lambda[ f_\epsilon g_\epsilon] \geq \mathbb{E}_\lambda [f_\epsilon]\cdot\mathbb{E}_\lambda [g_\epsilon] = \mathbb{P}_\lambda[A]\cdot \mathbb{P}_\lambda[B].
\]
Since the $\sigma$-algebra $\mathcal{F}_\epsilon$ converges for $\epsilon \downarrow 0$ upwards to the $\sigma$-algebra associated to the Poisson point process $\eta$, the left hand side converges by the martingale convergence theorem to $\mathbb{E}_\lambda [1_A1_B] = \mathbb{P}_\lambda[ A\cap B]$.
\end{proof}

Now, we only need to introduce Russo's formula for the Poisson-Boolean model.

\begin{defin}[Local events]
An event $A$ is called \emph{local}, if there exists a compact set $K\subseteq\mathbb{R}^d$ such that $A$ is $\Xi(K)$-measurable, where $\Xi(K)\subseteq\eta$ is the set of balls intersecting $K$, i.e.~ $\Xi(K) := \{(x,r_x)\in\eta\;\vert\; K\cap B_r^x\neq\emptyset\}$. In other words: the event only depends on balls that touch the compact $K$.
\end{defin}

\begin{theo}[Russo's formula]
Let $A$ be an increasing local event and define for $x\in\mathbb{Z}^d$ the random variable
\[
\mathrm{Piv}_{x,A} := 1_{\eta\not\in A}\int_{S_x}\int_{\mathbb{R}_+} 1_{\eta\cup\{(z,r)\}\in A} \d{\nu(r)}\d{z},
\]
where $S_x := x + [0,1)^d$ is the box induced by $x$. Then
\[
\dfrac{d}{d\lambda}\mathbb{P}_\lambda[A] = \mathbb{E}_\lambda\left[1_{\eta\not\in A}\int_{\mathbb{R}^d}\int_{\mathbb{R}_+} 1_{\eta\cup\{(z,r)\}\in A} \d{\nu(r)}\d{z}\right] = \sum_{x\in\mathbb{Z}^d} \mathbb{E}_\lambda[\mathrm{Piv}_{x,A}].
\]
\end{theo}
\begin{proof}
Let $\delta > 0$ be small enough. Then, consider two Poisson point process $\eta_\lambda$ and $\eta_\delta$ with intensity $\lambda\mathrm{d}z\otimes\nu$ and $\delta\mathrm{d}z\otimes\nu$ respectively. Write $P$ and $E$ for the joint probability measure and the associated expectation operator. Then
\[
\mathbb{P}_\lambda[A] = P(\eta_\lambda\in A)\quad\text{ and }\quad\mathbb{P}_{\lambda+\delta}[A] = P(\eta_\lambda\cup\eta_\delta\in A).
\]
Now, consider $\delta$ small enough such that $A$ only depends on the balls touching the compact ball $K_\delta := B(0, \delta^{-1/(3d)})$. We will soon use Mecke's formula (cf. Theorem \ref{theo:mecke}). To clarify how exactly we use it, we need to introduce some notation. Let $N_\delta$ be the (measurable) map sending a PPP to  the number of its balls touching $K_\delta$, i.e.
\[
N_\delta(\eta) := \vert \{(x, r_x)\in\eta\;\vert\; K_\delta\cap B(x,r_x)\neq\emptyset \}\vert.
\]
Furthermore, let $\Psi$ be the (measurable) map sending a PPP on the ball closest to the origin with respect to the Hausdorff distance. First, note that
\begin{align*}
P(N_\delta(\eta_\delta) \geq 2) &= 1 - \left(1 + \delta\int_{\mathbb{R}_+} (r + \delta^{-1/(3d)})^d\d{\nu(r)}\right)e^{-\delta\int_{\mathbb{R}_+} v_d\left(r + \delta^{-1/(3d)}\right)^d\d{\nu(r)}}\\
&\leq \left(\delta\cdot\int_{\mathbb{R}_+} v_d\left(r + \delta^{-1/(3d)}\right)^d\d{\nu(r)}\right)^2 \\
&= O\left(\left(O(\delta) + O(\delta^{2/3})\right)^2\right) = o(\delta),
\end{align*}
where $v_d$ is the volume of the $d$-dimensional unit ball. Thus, we may write 
\begin{align*}
\mathbb{P}_{\lambda+\delta}[A] - \mathbb{P}_\lambda[A] &= P(\eta_\lambda\not\in A\text{ and }\eta_\lambda\cup\eta_\delta\in A)\\
&= P(\eta_\lambda\not\in A, \eta_\lambda\cup\eta_\delta\in A\text{ and } N_\delta(\eta_\delta) \leq 1) + o\left(\delta\right)\\
&= P(\eta_\lambda\not\in A, \eta_\lambda\cup\eta_\delta\in A \text{ and } N_\delta(\eta_\delta) = 1) + o(\delta).
\end{align*}
Let us have a closer look at the first term. By writing
\[
\Phi(\eta_\lambda,\eta_\delta, z, r) := 1_{\eta_\lambda \cup \Psi(\eta_\delta)\in A}\cdot 1_{N_\delta(\eta_\delta) = 1}\cdot 1_{(z,r)\in S},
\]
with $S = \left\{(z,r)\;\vert\; r > 0, \vert z\vert \leq r + \delta^{-1/(3d)}\right\}$, we get, using $\eta_\delta(S) = N_\delta(\eta_\delta)$,
\begin{align*}
\mathbb{P}_{\lambda+\delta}[A] - \mathbb{P}_\lambda[A] &= E\left( 1_{\eta_\lambda\not\in A}\cdot E\left(\left.\int_S \Phi(\eta_\lambda, \eta_\delta, z, r) \d{\eta_d(z,r)}\;\right\vert\; \eta_\lambda\right)\right),
\end{align*}
where we used that, on the event $\{N_\delta(\eta_\delta) = 1\}$, only the ball closest to the origin can influence $A$. Now, $E(\cdot\;\vert\;\eta_\lambda)$ is simply the expectation with respect to $\eta_\delta$, which is a PPP. Hence, we may apply Mecke's formula (cf \ref{theo:mecke}):
\begin{align*}
E\left(\left.\int_S \Phi(\eta_\lambda,\eta_\delta, z, r) \d{\eta_\delta(z,r)}\;\right\vert\; \eta_\lambda\right) &= E\left(\left.\int_S \Phi(\eta_\lambda, \eta_\delta + \delta_{(z,r)}, z, r) \d{(\delta\mathrm{d}z\otimes\nu)(z,r)}\;\right\vert\;\eta_\lambda\right).
\end{align*}
Now, since $\Phi(\eta_\lambda, \eta_\delta + \delta_{(z,r)}, z, r)$ is nonzero only if $N_\delta(\eta_\delta) = 0$, this yields
\begin{align*}
E\left(\delta\left.\int_{\mathbb{R}_+} \int_{B(0,r+\delta^{-1/(3d)})} 1_{\eta_\lambda\cup\{(z,r)\}\in A}\cdot 1_{N_\delta(\eta_\delta) = 0} \d{z}\d{\nu(r)} \;\right\vert\;\eta_\lambda\right)
\end{align*}
for the above, and thus
\begin{align*}
\mathbb{P}_{\lambda+\delta}[A] - \mathbb{P}_\lambda[A] &= \delta\cdot E\left(1_{N_\delta(\eta_\delta) = 0}\cdot 1_{\eta_\lambda\not\in A}\cdot \int_{\mathbb{R}_+}\int_{B(0,r+\delta^{-1/(3d)})} 1_{\eta_\lambda \cup \{(z,r)\}\in A} \d{z}\d{\nu(r)}\right).
\end{align*}
Now, the integral part converges monotonously to the integral over the entire space. Furthermore, 
\[
P(N_\delta(\eta_\delta) = 0) = e^{-\delta\int_{\mathbb{R}_+} v_d(r + \delta^{-1/(3d)})^d\d{\nu(r)}} \underset{\delta\downarrow 0}{\longrightarrow} 1,
\]
hence the corresponding indicator function converges almost surely. We conclude by dominated convergence that
\begin{align*}
\lim_{\delta\downarrow 0} \dfrac{\mathbb{P}_{\lambda+\delta}[A] - \mathbb{P}_\lambda[A]}{\delta} &= E\left(1_{\eta_\lambda\not\in A}\int_{\mathbb{R}^d}\int_{\mathbb{R}_+} 1_{\eta_\lambda\cup\{(z,r)\}\in A} \d{\nu(r)}\d{z}\right)\\
& = \mathbb{E}_\lambda\left[1_{\eta_\lambda\not\in A}\int_{\mathbb{R}^d}\int_{\mathbb{R}_+} 1_{\eta_\lambda\cup\{(z,r)\}\in A} \d{\nu(r)}\d{z}\right].
\end{align*}
The left limit is obtained analogously.
\end{proof}

\subsection{Applying the OSSS Method to Poisson-Boolean Percolation}

The overall method is identical to the one we used in the discrete case. With Russo's formula, we obtain a differential inequality which leads to the result. The important difference is that we will not get the same strong result as in the discrete case. It has already been shown that the subcritical phase does generally not exhibit an exponential decay. As mentioned before, we will attack the very different problem of proving $\lambda_c = \widehat{\lambda}_c$. As before, we will always assume $d \geq 2$.

\begin{theo}[Duminil-Copin, Raoufi, Tassion; 2018]\label{theo:DRT18bool}
Assume
\begin{equation}\label{eq:moment_assumption}
\int_{\mathbb{R}_+} r^{5d - 2} \d{\nu(r)}< +\infty.
\end{equation}
Then, we have that $\lambda_c = \widehat{\lambda}_c$. Furthermore, there exists $c > 0$ such that $\theta(\lambda) \geq c(\lambda - \lambda_c)$ for all $\lambda \geq \lambda_c$.
\end{theo}

The proof if much more complex than in the discrete case, because we have to cope with a long range dependence in the Poisson-Boolean model. I choose to present the intermediate results in an order different from the one in the original article. This gives me the opportunity to give more detailed insights in the ideas behind the proof. Before we begin the proof of the theorem though, we need an analogous lemma to Lemma \ref{lem:diff_ineq}. Since the same arguments apply \emph{mutatis mutandis} to this variant, we will omit the proof.

\begin{lem}\label{lem:diff_ineq_cont}
If there exists a constant $c > 0$ such that, for all $r \geq 0$ and $\lambda\geq \widehat{\lambda}_c$,
\begin{equation}\label{eq:diff_ineq_cont}
\theta_r' \geq c\dfrac{r}{\Sigma_r}\theta_r(1-\theta_r),
\end{equation}
where $\Sigma_r := \int_0^r \theta_s\d{s}$, then for every $\lambda_0 > \widehat{\lambda}_c$, there exists a $\lambda_1\in [\widehat{\lambda}_c,\lambda_0]$ such that:
\begin{enumerate}
	\item For any $\lambda \in (\widehat{\lambda}_c,\lambda_1)$, there exists $c_\lambda > 0$ such that
	\[
	\theta_r \leq \exp(-c_\lambda r).
	\]
	\item For all $\lambda \in [\lambda_1,\lambda_0]$, one has $\theta(\lambda) \geq c(\lambda - \lambda_1)$.
\end{enumerate}
\end{lem}

Note that if the conditions of the lemma are satisfied, then $\lambda_c = \widehat{\lambda}_c$. Indeed suppose $\lambda_0 := \lambda_c > \widehat{\lambda}_c$. Then, for every $\lambda \in (\widehat{\lambda}_c,\lambda_c)$,
\[
\mathbb{P}_\lambda[S_r\leftrightarrow S_{2r}] \leq \theta_{r/2} \leq \exp(-c_\lambda r/2) \underset{r\to+\infty}{\longrightarrow} 0,
\]
which contradicts the definition of $\widehat{\lambda}_c$. Hence, it suffices to prove that the differential inequality \eqref{eq:diff_ineq_cont} is satisfied.

\begin{proof}[Proof of Theorem \ref{theo:DRT18bool}]
To use the OSSS inequality, we need to describe the PPP as a product space. We will do this by writing our space $\mathbb{R}^d$ as the disjoint union $\mathbb{R}^d = \bigsqcup_{x\in\mathbb{Z}^d} S_x$, where $S_x$ is the box $x + [0,1)^d$. Then write $\eta_{(x,n)} := \eta\cap (S_x\times [n,n+1))$ for all $x\in\mathbb{Z}^d$ and $n\in\mathbb{N}$. These sections of $\eta$ are all independent, because the considered sets are pairwise disjoint. 

Note that we only proved the weak version of the OSSS inequality for a \emph{finite} product of probability spaces. Hence, take $L \in\mathbb{N}$ and write
\[
I_L := \{ (x,n)\in\mathbb{Z}^d\times\mathbb{N}\;\vert\; \Vert x\Vert\leq L \text{ and }n \leq L\}.
\]
We integrate all the other indices in one additional space by writing $\eta_{other}$ for the union of all the $\eta_i$ with $i\not\in I_L$. This induces the product space $\Omega = \left(\prod_{i\in I_L} \Omega_i\right)\times \Omega_{other}$ and the corresponding product measure.

To simplify the reading, we will write $A := \{0\leftrightarrow \partial B_r\}$. For $f := 1_A$, the OSSS inequality then implies
\[
\theta_r(1-\theta_r) \leq \sum_{i\in I_L} \delta_i(T)\mathrm{Inf}_i(f) + \delta_{other}(T)\mathrm{Inf}_{other}(f),
\]
where $T$ is an algorithm determining $f$. Since
\[
\mathrm{Inf}_{other}(f) \leq \mathbb{P}_\lambda[\exists (z,R)\in\eta,\; R > L\text{ and } B_R^z\cap B_r \neq \emptyset] \underset{L\to+\infty}{\longrightarrow} 0.
\]
Using $\delta_{other}(T) \leq 1$, we may extend the OSSS inequality in this context to the infinite index set $\mathbb{Z}^d\times\mathbb{N}$:
\[
\theta_r(1-\theta_r) \leq \sum_{x\in\mathbb{Z}^d} \sum_{n\in\mathbb{N}} \delta_{(x,n)}(T)\mathrm{Inf}_{(x,n)}(f).
\]
The next step will be to compare the right hand side to the derivate of $\theta_r$. For this, we will relate the influence of a cell $S_x\times [n,n+1)$ to the quantity $\mathbb{E}_\lambda[\mathrm{Piv}_{x,A}]$. First, note that
\begin{align*}
\mathrm{Inf}_{x,n}(f) &= \mathbb{P}_\lambda[1_A(\eta) \neq 1_A(\widetilde{\eta})] \leq 2\mathbb{P}_\lambda[\eta\not\in A,\widetilde{\eta}\in A],
\end{align*}
where $\widetilde{\eta}$ is equal to $\eta$ except for the cell $(x,n)$ where it is resampled independently. The event on the right hand side implies that $\widetilde{\eta}\cap  S_x\times [n,n+1)$ contains at least one point. For $\eta$ not to occur, this means that the event $\mathcal{P}_x(n+1+\sqrt{\delta})$ has to occur for $\eta$, where
\[
\mathcal{P}_x(n) := \{0 \leftrightarrow B_n^x\}\cap \{B_n^x\leftrightarrow \partial B_r\}\cap \{0\not\leftrightarrow \partial B_r\}.
\]
(See Figure \ref{fig:P(n)} for a visualisation.)

\begin{figure}[ht]
	\centering
	\includegraphics[width=0.7\textwidth]{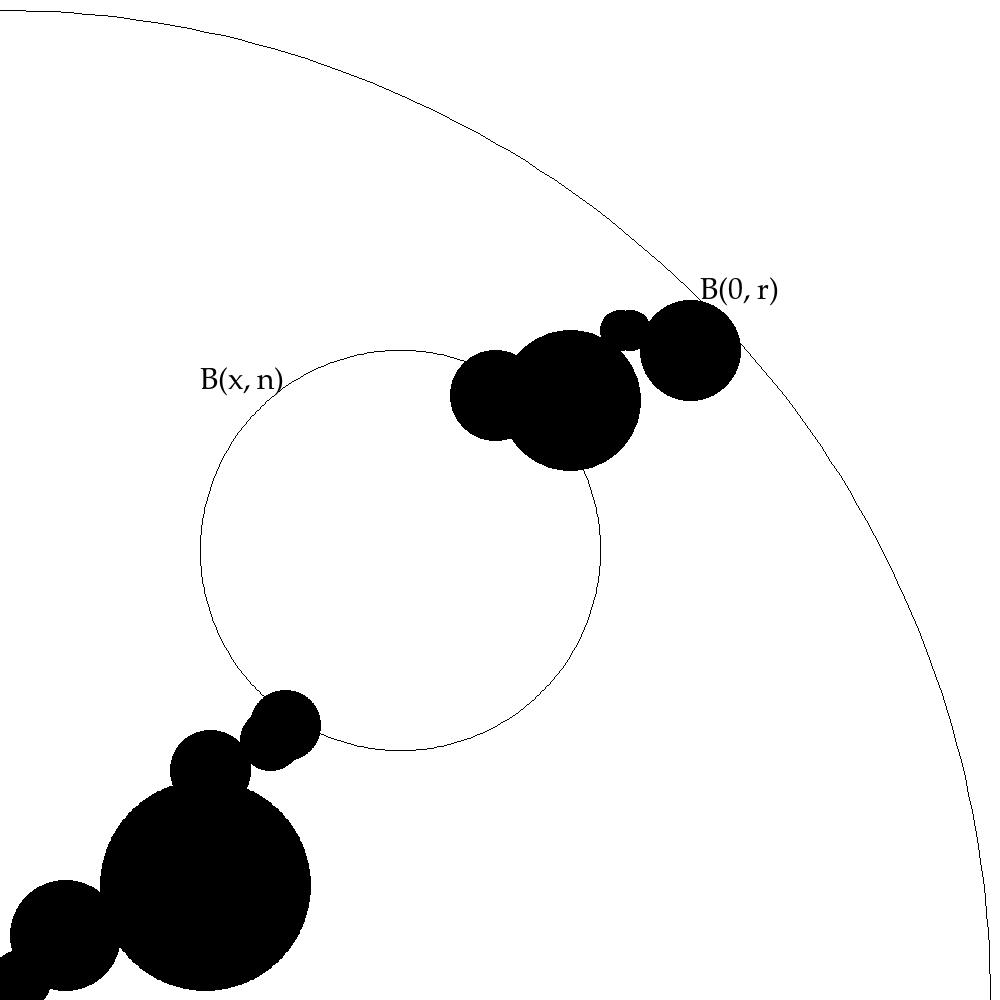}
	\caption{Visualisation of the event $\mathcal{P}(n)$}
	\label{fig:P(n)}
\end{figure}

This means that
\begin{align*}
\mathrm{Inf}_{(x,n)}(f) &\leq 2\mathbb{P}_\lambda[\vert\widetilde{\eta}\cap (S_x\times[n,n+1))\vert \geq 1]\cdot\mathbb{P}_\lambda[\eta\in\mathcal{P}_x(n+1+\sqrt{d})] \\
&\leq 2\lambda\cdot\nu[n,n+1)\cdot\mathbb{P}_\lambda[\mathcal{P}_x(n+1+\sqrt{d})]
\end{align*}
by independence.

The problem with this bound is that it still depends in an implicit way on $n$. First, let us try to relate the events defined by $\mathcal{P}_x$ to $\mathrm{Piv}_{x,A}$. Let $r_* > 0$ and $r^* > r_*$ be such that the following "insertion tolerance property"
\[
c_{IT} := c_{IT}(\lambda) := \mathbb{P}_\lambda[\mathcal{D}_x] > 0
\] 
is satisfied for some and thus for all $x\in\mathbb{Z}^d$, where
\[
\mathcal{D}_x := \{\;\exists (z,R)\in\eta,\quad z\in S_x\text{ and } B_{r_*}^x\subseteq B_R^z \subseteq B_{r^*}^x\;\}.
\]
In other words, it is possible to find a ball near $x$ which contains a small ball around $x$, but which is not too big (cf. Figure \ref{fig:D_x}). 

\begin{figure}[ht]\centering
\begin{tikzpicture}[scale=0.8]
	
	\draw[color=cyan!40, fill=cyan!40] (0.4, 0.6) circle (4);

	\draw (0,0) node{$\bullet$};
	\draw (0,0) node[below right]{$x$} ;
	\draw (1,1) -- ++(0,-2) -- ++(-2,0) -- ++(0,2) -- (1,1);
	\draw (0,0) circle (2);
	\draw (0,0) circle (7);
	
	\draw (0.4, 0.6) node{$\bullet$};
	\draw (0.4, 0.6) node[left]{$(z,R)$};
	
	\draw (45 : 1.9) node[above right]{$B_{r_*}^x$};
	\draw (45 : 4.5) node[above right]{$B_R^z$};
	\draw (45 : 6.9) node[above right]{$B_{r^*}^x$};

\end{tikzpicture}
\caption{Visualisation of the event $\mathcal{D}_x$}
\label{fig:D_x}
\end{figure}
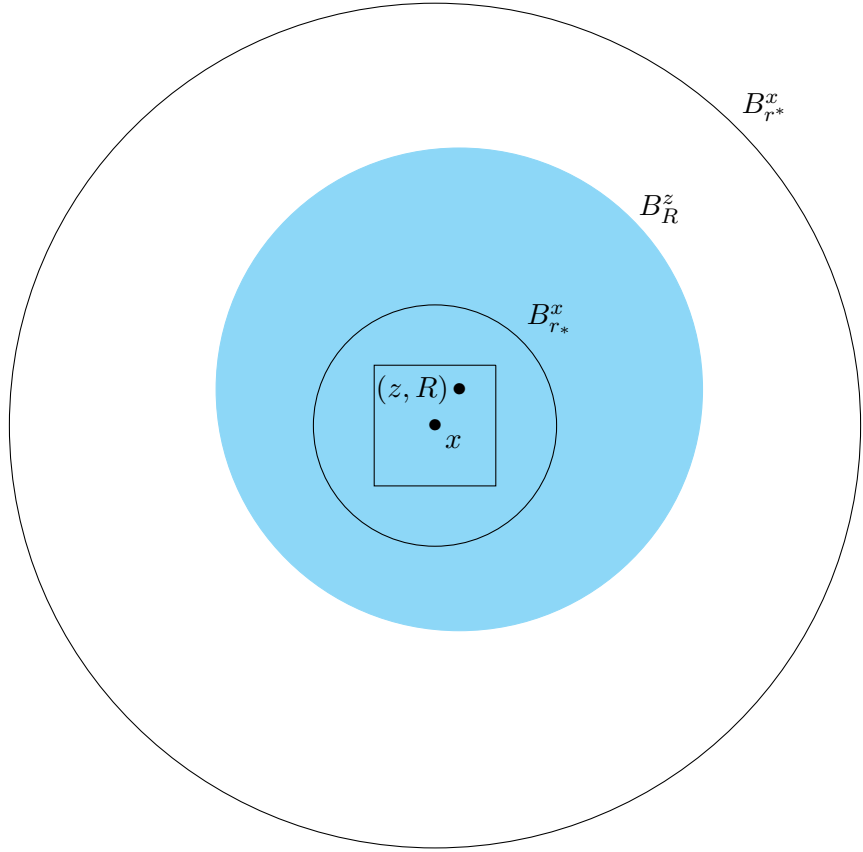


Denote by $\widetilde{\eta}$ an independent PPP with same intensity $\lambda\mathrm{d}z\otimes\nu$ on $\mathbb{R}^d\times\mathbb{R}_+$. If $P$ denotes the joint probability measure and if $(\zeta,R)$ denotes on $\mathcal{D}_x$ the closest point to the origin which satisfies $\mathcal{D}_x$, then
\begin{align*}
c_{IT}\mathbb{P}_\lambda[\mathcal{P}_x(r_*)] &= P(\widetilde{\eta}\in\mathcal{D}_x\text{ and } \eta\in \mathcal{P}_x(r_*))\\
&\leq P(\eta\not\in A\text{ and } \eta\cup\{(\zeta,R)\}\in A, \widetilde{\eta}\in\mathcal{D}_x).
\end{align*}
If this last event occurs, then there exists at least some ball $(z,r)\in \widetilde{\eta}\cap( S_x\times\mathbb{R}_+)$ such that $\eta\not\in A$, but $\eta\cup\{(z,r)\}\in A$. This means that
\[
1_{\eta\not\in A\text{ and } \eta\cup\{(\zeta,R)\}\in A, \widetilde{\eta}\in\mathcal{D}_x} \leq 1_{\eta\not\in A}\cdot1_{\exists (z,r)\in \widetilde{\eta}\cap (S_x\times\mathbb{R}_+),\, \eta\cup\{(z,r)\}\in A} \leq 1_{\eta\not\in A}\sum_{(z,R)\in \widetilde{\eta}} 1_{\eta\cup\{(z,r)\}\in A}.
\]
The right hand side is exactly the integral with respect to $\widetilde{\eta}$. Hence, Mecke's formula gives
\begin{align*}
c_{IT}\mathbb{P}_\lambda[\mathcal{P}_x(r_*)]&\leq \mathbb{E}\left(\lambda\cdot 1_{\eta\not\in A}\int_{S_x}\int_{\mathbb{R}_+}  1_{\eta\cup\{z,r\}\in A}\d{\nu(r)}\d{z}\right)\\
&= \lambda\mathbb{E}_\lambda[\mathrm{Piv}_{x,A}].
\end{align*}
The calculation is possible for all values of $r_*$, but the quantity $c_{IT}$ is not constant. If $r_*$ becomes big, $c_{IT}$ tends to $0$ and the bound becomes useless. The dependence in $\lambda$ is much less worrying, since
\[
c_{IT}(\lambda_c) \leq c_{IT}(\lambda) \leq c_{IT}(\lambda_0)
\]
for some $\lambda_0 \geq \lambda_c$ and all $\lambda\in [\lambda_c,\lambda_0]$. Since our result is only local, we can choose some fixed $\lambda_0$ large enough and consider $c_{IT}$ to be bounded away from $0$ and infinity. 

For simplicity, we fixed the boxes $S_x$ to have size 1. By a scaling argument, we can ask for
\begin{equation}\label{eq:r**}
1 + 2\sqrt{d} \leq r_*  \leq r^*\leq 2r_* - 2\sqrt{d}.
\end{equation}

The main step is the most difficult one. To get rid of the implicit dependency in $n$ and to obtain the expression of the derivative of $\mathbb{P}_\lambda[A]$, we need to relate $\mathcal{P}_x(n)$ to $\mathcal{P}_x(r_*)$. First, define $\mathbf{C} := \mathbf{C}(\eta) := \{z\in\mathbb{R}^d\;\vert\; 0\leftrightarrow z\}$ and $\mathbf{C}' := \mathbf{C}'(\eta) := \{z\in\mathbb{R}^d\;\vert\; z\leftrightarrow \partial B_r\}$ the connected components of $0$ and $\partial B_r$ respectively. We claim that
\begin{equation}\label{eq:claim_bool}
\mathbb{P}_\lambda[\mathcal{P}_x(n)\text{ and }d(\mathbf{C}\cap B_{3n}^x, \mathbf{C}') < r^*] \geq \dfrac{c_1}{n^{3d-2}}\mathbb{P}_\lambda[\mathcal{P}_x(n)]
\end{equation}
for some constant $c_1 > 0$. (In the following, all $c_i$ will be positive constants.) This would provide the conclusion of the proof: Choose some $z\in\mathbb{R}^d$ at half distance between $\mathbf{C}\cap B_{3n}^x$ and $\mathbf{C}'$. Take some $y\in\mathbb{Z}^d$ at most $\sqrt{d}$ from $z$. From the last inequality in \eqref{eq:r**}, it follows that $y$ is at most $r_*$ from $\mathbf{C}\cap B_{3n}^x$ and $\mathbf{C}'$. Apparently, $\mathcal{P}_y(r_*)$ must occur. This means that
\[
\mathcal{P}_x(n) \leq c_2n^{3d-2} \sum_{y\in\mathbb{Z}^d\cap B_x^{3n+r_*}} \mathbb{P}_\lambda[\mathcal{P}_y(r_*)]
\]
and thus
\begin{align*}
\mathrm{Inf}_{(x,n)}(f) &\leq 2\lambda c_3\nu[n,n+1)\cdot (n+1)^{3d-2}\cdot (n+1)^d\cdot \dfrac{\lambda}{c_{IT}} \mathbb{E}_\lambda[\mathrm{Piv}_{x,A}]\\
&\leq c_4(n+1)^{4d - 2}\nu[n,n+1)\cdot\mathbb{E}_\lambda[\mathrm{Piv}_{x,A}],
\end{align*}
where we used again, that $\lambda \leq \lambda_0$ is bounded by a fixed value. All together, this gives
\begin{align*}
\theta_r(1-\theta_r) \leq c_4\sum_{n\in\mathbb{N}} \delta_{(x,n)}(T)(n+1)^{4d-2}\nu[n,n+1)\sum_{x\in\mathbb{Z}^d}\mathbb{E}_\lambda[\mathrm{Piv}_{x,A}]
\end{align*}
and all we have to do is to bound the revealment of our algorithms. This very last step is very similar to the discrete case so that we will look at it only at the end of the proof.

Let us concentrate on the claim \eqref{eq:claim_bool} first. To prove it, we need to work a bit.
\begin{lem}
There exists some constant $c_5 > 0$ such that for all $\lambda > \widehat{\lambda}_c$ and all $r \geq r_*$,
\[
\forall x\in\partial B_r,\qquad \mathbb{P}_\lambda[0 \overset{B_r}{\leftrightarrow} B_{r^*}^x] \geq \dfrac{c_5}{r^{2d-2}}.
\]
\end{lem}
\begin{proof}
First, recall from Theorem \ref{theo:open_set} that the set $\{\lambda > 0\;\vert\; \lim_{r\to+\infty}\mathbb{P}_\lambda[S_r\leftrightarrow S_{2r}] = 0\}$ is open. That means that there exists some $c_6 > 0$ such that for all $\lambda > \widehat{\lambda}_c$ and all $r \geq r_*$,
\[
\mathbb{P}_\lambda[S_r\leftrightarrow S_{2r}] \geq c_6.
\]
By union bound, this means that
\[
c_6 \leq \sum_{x \in X} \mathbb{P}_\lambda[S_x\leftrightarrow \partial B_{2r}] \leq \vert X\vert \cdot \mathbb{P}_\lambda[S_0 \leftrightarrow \partial B_r],
\]
where $X$ is the set of $x\in\mathbb{Z}^d$ such that $S_x\cap \partial B_r \neq\emptyset$. Since $\vert X\vert$ is proportional to $r^{d-1}$, we get the bound
\begin{equation}\label{eq:first_bound_S_0_B_r}
\mathbb{P}_\lambda[S_0\leftrightarrow \partial B_r] \geq \dfrac{c_7}{r^{d-1}}
\end{equation}
for some $c_7 > 0$ and all $\lambda > \widehat{\lambda}_c$ and all $r \geq r_*$. This is the only part in the entire proof that we require $\lambda > \widehat{\lambda}_c$.

The rest of the proof deals with the problem that a path could use balls from outside $B_r$. This should not be very probable, because of the moment assumption \eqref{eq:moment_assumption}. To simplify the problem, we will prove that there exist some different $r^{**} > 0$ we may choose as we wish and some $c_8 > 0$ such that 
\[
\mathbb{P}_\lambda[S_0 \overset{B_r}{\leftrightarrow} B_{r^{**}}^x] \geq \dfrac{c_8}{r^{2d-2}}.
\]
Indeed, using 
\[
\{0\overset{B_r}{\leftrightarrow} B_{r^*}^x\} \supseteq \mathcal{D}_0 \cap \{S_0\overset{B_r}{\leftrightarrow} B_{r^{**}}^x\} \cap \bigcap_{t\in \mathcal{T}} \mathcal{D}_t,
\]
where $\mathcal{T}$ is the set of all $t\in\mathbb{Z}^d$ such that $S_t\cap B_{2r^{**}}^x\neq \emptyset$, the FKG inequality yields
\[
\mathbb{P}_\lambda[0\overset{B_r}{\leftrightarrow} B_{r^*}^x] \geq {c_{IT}}^{\vert \mathcal{T}\vert + 1}\cdot \mathbb{P}_\lambda[S_0 \overset{B_r}{\leftrightarrow} B_{r^{**}}^x].
\]

Now, we will attack the central problem and how to separate the external noise of balls not entirely contained in $B_r$. First, fix $r^{**} \geq 2r^*$ large enough (see the end of the proof). Then, divide the $B_r$ into the inner balls $B_{r_*}^y$ with $y\in Y := (\mathbb{Z}^d\cap B_{r-r^{**}})\setminus\{0\}$ and the boundary balls $B_{r^{**}}^z$ where $z\in Z$ and $Z\subseteq\partial B_r$ is finite but sufficiently large such that 
\[
B_r \subseteq \bigcup_{y\in Y} B_{r_*}^y \cup\bigcup_{z\in Z} B_{r^{**}}^z.
\]

\begin{figure}[!htbp]
\centering
	\begin{subfigure}[b]{0.72\linewidth}
		\centering
		\includegraphics[width=0.9\linewidth]{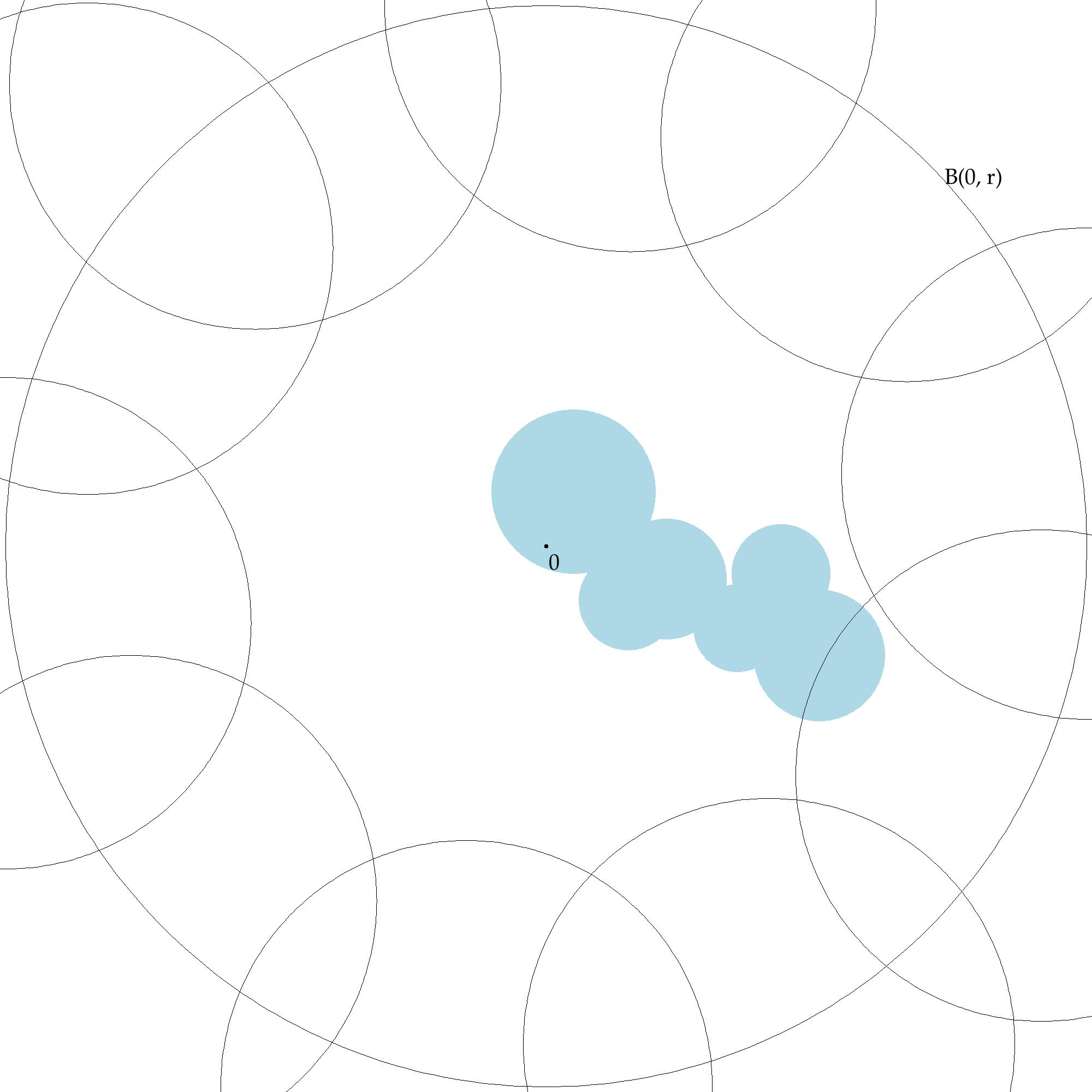}
		\caption{Favorable case}
	\end{subfigure}
	\begin{subfigure}[b]{0.72\linewidth}
		\centering
		\includegraphics[width=0.9\linewidth]{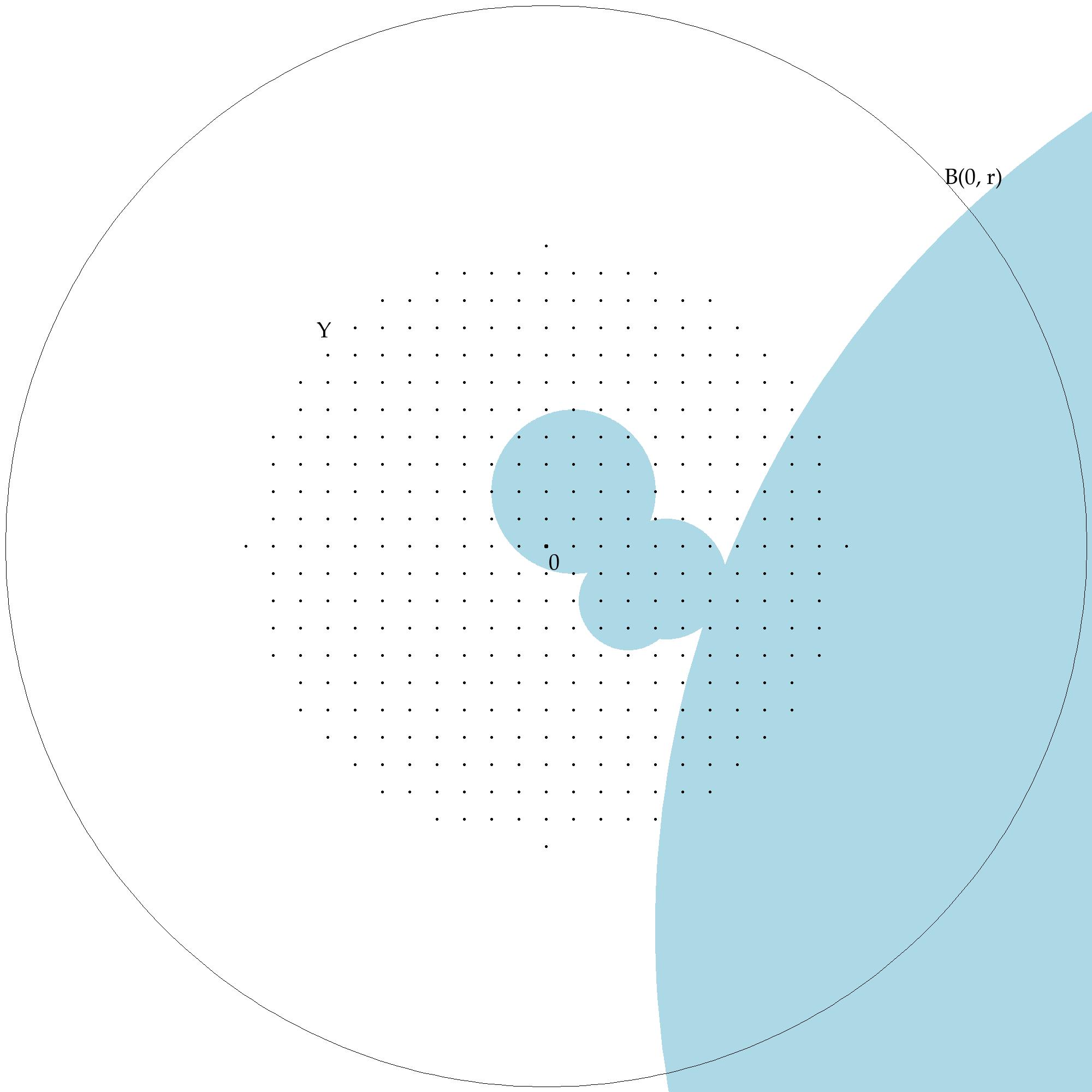}
		\caption{Unfavorable case}
	\end{subfigure}
	\caption{On the event $\{S_0 \leftrightarrow \partial B_r\}$ two things may happen}
	\label{fig:event_S_0_B_r}
\end{figure}
If $S_0$ is connected to $\partial B_r$, then there exists some $z\in Z$ such that $S_0$ is connected to $B_{r^{**}}^z$ in $B_r$ or there exists some $y\in Y$ such that the event
\[
\mathcal{A}_y := \{S_0\overset{B_r}{\leftrightarrow} B_{r_*}^y\}\cap \{\exists (u,R)\in\eta\text{ such that $B_R^u$ intersects both $B_{r_*}^y$ and $\partial B_r$}\}
\]
occurs (see Figure \ref{fig:event_S_0_B_r}), i.e.
\[
\sum_{z\in Z}\mathbb{P}_\lambda[S_0\overset{B_r}{\leftrightarrow} B_{r^{**}}^z] + \sum_{y\in Y} \mathbb{P}_\lambda[\mathcal{A}_y] \geq \mathbb{P}_\lambda[S_0\leftrightarrow \partial B_r] \geq \dfrac{c_7}{r^{d-1}}.
\]
From the moment assumption \eqref{eq:moment_assumption} on $\nu$, it is clear that the second term on the left hand side should be treated as an error term. Note that if it was not there, the result would follow from the fact that $\vert Z\vert \leq cr^{d-1}$ for some $c > 0$. For a first bound on $\mathbb{P}_\lambda[\mathcal{A}_y]$, we will only use the independence between the defining events and the fact that for $y\in Y$, every point in $B_{r_*}^y$ is at least at distance $r - \Vert y\Vert - r_*$ from $\partial B_r$. In the following, we fix $y\in Y$ and write $\overline{r} := r - \Vert y \Vert$. Then
\[
\mathbb{P}_\lambda[\mathcal{A}_y] \leq \mathbb{P}_\lambda[\vert\eta\cap V\vert \geq 1] \cdot \mathbb{P}_\lambda[S_0\overset{B_r}{\leftrightarrow} B_{r_*}^y],
\]
where $V := \{(u,R)\in \mathbb{R}^d\times\mathbb{R}_+\;\vert\; R \geq \frac{\overline{r} - r_*}{2}\text{ and } u\in B_{r_* + R}^y\}$. It is
\begin{align*}
\mathbb{P}_\lambda[\vert\eta\cap V\vert \geq 1] &\leq c_9\int_{\frac{\overline{r} - r_*}{2}}^{+\infty} (r_* + \rho)^d \d{\nu(\rho)}.
\end{align*}
Using the fact that $\Vert y\Vert \leq r - r^{**}$ and $r^{**} \geq 2r^*$, we get
\[
\overline{r} - r_* \geq \dfrac{\overline{r}}{2} + \dfrac{\overline{r}}{2} - r^* \geq \dfrac{\overline{r}}{2} + r^* - r^* \geq \dfrac{\overline{r}}{2} \geq r^* \geq r_*.
\]
Then, the moment assumption on $\nu$ yields
\[
\mathbb{P}_\lambda[\vert\eta\cap V\vert\geq 1]\leq c_{10} \int_{\overline{r}/2}^{+\infty} \rho^d\d{\nu(\rho)} \leq \dfrac{c_{11}}{{\overline{r}}^{4d-3}}.
\]
(Note that we do not formally need such a strong moment assumption for this lemma.) All together, 
\[
\mathbb{P}_\lambda[\mathcal{A}_y] \leq \dfrac{c_{11}}{\overline{r}^{4d-3}}\cdot\mathbb{P}_\lambda[S_0\overset{B_r}{\leftrightarrow} B_{r_*}^y].
\]
We still have to control the second term. For this, we will relate it to the quantity of interest $\mathbb{P}_\lambda[S_0\overset{B_r}{\leftrightarrow} B_{r^{**}}^x]$. Note that for $w := ry/\Vert y\Vert$, 
\[
\{S_0\overset{B_r}{\leftrightarrow} B_{r^{**}}^w\} \supseteq \{S_0\overset{B_r}{\leftrightarrow} B_{r_*}^y\}\cap D_y\cap \{S_y\overset{B_{\overline{r}}^y}{\leftrightarrow} B_{r^{**}}^w]
\]
and the FKG inequality yields
\[
\mathbb{P}_\lambda[S_0\overset{B_r}{\leftrightarrow} B_{r^{**}}^w] \geq \mathbb{P}_\lambda[S_0\overset{B_r}{\leftrightarrow} B_{r_*}^y]\cdot c_{IT}\cdot\mathbb{P}_\lambda[S_y\overset{B_{\overline{r}}^y}{\leftrightarrow} B_{r^{**}}^w].
\]
Note that the quantity $U(r) := \mathbb{P}_\lambda[S_0 \overset{B_r}{\leftrightarrow} B_{r^{**}}^x]$ is independent of the choice of $x\in \partial B_r$. Hence, the inequality above gives
\[
\mathbb{P}_\lambda[\mathcal{A}_y] \leq c_{12}\dfrac{U(r)}{U(\overline{r})\cdot\overline{r}^{4d-3}}.
\]
Together with our first union bound, we obtain the recursive formula
\[
c_{13} r^{d-1} U(r) + c_{12}\sum_{y\in Y} \dfrac{U(r)}{U(r - \Vert y\Vert)\cdot (r - \Vert y \Vert)^{4d-3}} \geq  \dfrac{c_7}{r^{d-1}}.
\]
It suffices now to prove the property by induction. For this, it will become necessary to choose $r^{**}$ large enough. First, note that $U(r) \equiv 1$ on $[0, r^{**}]$. Since $\Vert y\Vert \geq 1$ for all $y\in Y$, we will consider an interval of the form $[k, k+1]$ with $k \geq r^{**}$ and suppose that the result is proven for all $r\leq k$. Now, take $r\in [k, k+1]$. Then
\[
\dfrac{1}{U(r - \Vert y\Vert)\cdot (r - \Vert y\Vert)^{4d - 3}} \leq \dfrac{(r - \Vert y\Vert)^{2d-2}}{c_8(r - \Vert y\Vert)^{4d-3}} = \dfrac{1}{c_8(r - \Vert y\Vert)^{2d-1}}.
\]
Note that
\[
\sum_{y\in Y} \dfrac{1}{(r - \Vert y\Vert)^{2d-1}} \leq \int_0^{r-r^{**}} \dfrac{\rho^{d-1}}{(r - \rho)^{2d-1}}\d{\rho} \leq \dfrac{r^{d-1}}{{r^{**}}^{2d-2}}.
\]
Choosing $r^{**}$ large enough so that $c_{14} := \frac{1}{{r^{**}}^{2d-2}} < \frac{c_7}{2c_{12}}$ finally yields
\[
\dfrac{c_7}{r^{d-1}} \leq U(r) r^{d-1}\cdot \left(c_{13} + \frac{c_7}{2c_8}\right)
\]
and hence
\[
U(r) \geq \dfrac{1}{c_{13}c_8/c_7 + 1/2} \cdot\dfrac{c_8}{r^{2d-2}}.
\]
Using the fact that we can choose $c_8$ arbitrarily small, we may assume the first fraction to be greater than 1, giving
\[
\mathbb{P}_\lambda[S_0 \overset{B_r}{\leftrightarrow} B_{r_*}^x] = U(r) \geq \dfrac{c_8}{r^{2d-2}}.
\]
Together with the arguments from the beginning, we get the result
\[
\mathbb{P}_\lambda[0 \overset{B_r}{\leftrightarrow} B_{r^*}^x] \geq \dfrac{c_5}{r^{2d-2}}
\]
for some constant $c_5 > 0$.
\end{proof}

We are now able to attack the proof of our claim \eqref{eq:claim_bool}. First, we will see the intuitive argument which will be supported by an approximation argument afterwards.  First, fix $y\in \mathbb{Z}^d$ such that $S_y\cap B_n^x\neq\emptyset$. For some subset $C\subseteq\mathbb{R}^d$, we define $u_C := u_C(y)$ as the point of $S_y$ farthest away from $C$ and $v_C :=v_C(y)$ the point of $C$ closest to $u_C$. If more than one point satisfies the conditions, choose one with respect to some fixed ordering.

Consider the event
\[
\mathcal{E}_y := \{ \mathbf{C}\cap B_n^x\neq\emptyset\text{ and }d(u_\mathbf{C},\mathbf{C}) \geq r^*\}
\]
and note that it is $\mathbf{C}$-measurable.  Then, define the events
\[
\mathcal{F}_y := \mathcal{D}_{u_\mathbf{C}},\quad \mathcal{G}_y := \{u_\mathbf{C} \overset{\mathbb{R}^d\setminus\mathbf{C}}{\leftrightarrow} B_{r^*}^{v_\mathbf{C}}\},\quad \mathcal{H}_y := \{ S_y\cap B_n^x\overset{\mathbb{R}^d\setminus\mathbf{C}}{\leftrightarrow} \partial B_r\}.
\]

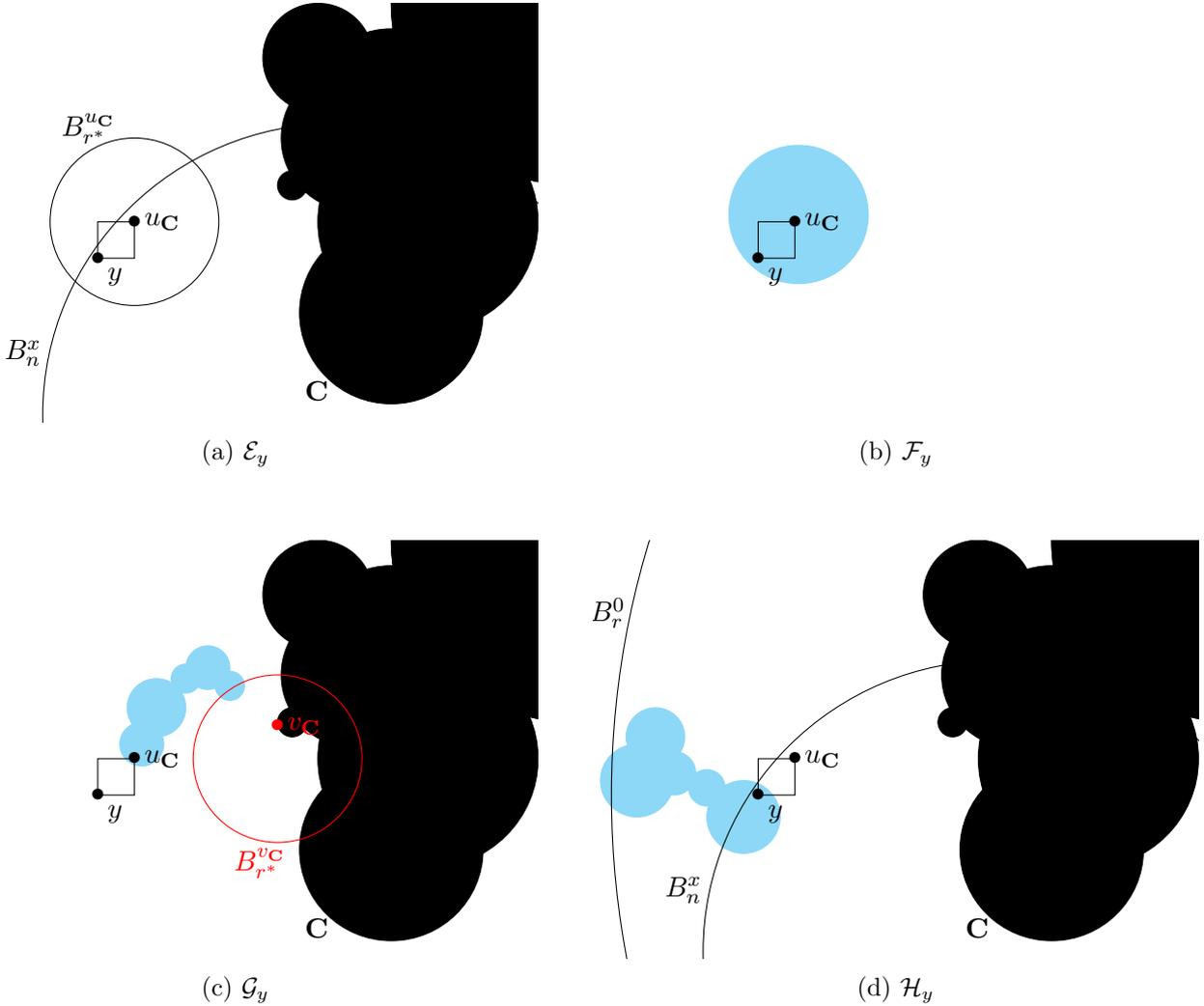
\begin{figure}[H]
	\centering
	\hspace*{-0.9in}
	\begin{subfigure}[b]{0.6\linewidth}
		\centering
		\begin{tikzpicture}[scale=0.5]

	\clip (-4.5, -4.5) rectangle (12, 7);
	
	\draw (0,0) node{$\bullet$};
	\draw (0,0) node[below right]{$y$};
	\draw (0,0) -- (1,0) -- (1,1) -- (0,1) -- (0,0);
	\draw (1,1) node{$\bullet$};
	\draw (1,1) node[right]{$u_{\mathbf{C}}$};
	
	\draw (1,1) circle (2.3);
	\draw (95: 2.9) node[above]{$B_{r^*}^{u_{\mathbf{C}}}$};
	
	\draw (6.5,-4.3) circle (8);
	\draw (-116:2.9) node[left]{$B_n^x$};
	
	\draw[fill=black] (13, 7) circle (5);
	\draw[fill=black] (9, 1) circle (3);
	\draw[fill=black] (7, 3.3) circle (2);
	\draw[fill=black] (13, 7) circle (5);
	\draw[fill=black] (8, -1.5) circle (2.5);
	\draw[fill=black] (8, 4) circle (2.3);
	\draw[fill=black] (6, 5.5) circle (1.5);
	\draw[fill=black] (5.3, 2) circle (0.4);
	
	\draw (-31.5:7) node{$\mathbf{C}$};

\end{tikzpicture}
		\caption{$\mathcal{E}_y$}
	\end{subfigure}%
	\begin{subfigure}[b]{0.6\linewidth}
		\centering
		\begin{tikzpicture}[scale=0.5]

	\clip (-4.5, -4.5) rectangle (12, 7);
	
	\draw[color=cyan!40, fill=cyan!40] (1.1, 1.2) circle (1.9);

	\draw (0,0) node{$\bullet$};
	\draw (0,0) node[below right]{$y$};
	\draw (0,0) -- (1,0) -- (1,1) -- (0,1) -- (0,0);
	\draw (1,1) node{$\bullet$};
	\draw (1,1) node[right]{$u_{\mathbf{C}}$};

\end{tikzpicture}
		\caption{$\mathcal{F}_y$}
	\end{subfigure}\hfill\\\hfill\\\hfill\\
	\hspace*{-0.9in}
	\begin{subfigure}[b]{0.6\linewidth}
		\centering
		\begin{tikzpicture}[scale=0.5]

	\clip (-4.5, -4.5) rectangle (12, 7);

	\draw[color=cyan!40, fill=cyan!40] (1.2, 1.4) circle (0.6);
	\draw[color=cyan!40, fill=cyan!40] (1.6, 2.4) circle (0.8);	
	\draw[color=cyan!40, fill=cyan!40] (2.4, 3.2) circle (0.4);
	\draw[color=cyan!40, fill=cyan!40] (3, 3.5) circle (0.6);
	\draw[color=cyan!40, fill=cyan!40] (3.6, 3) circle (0.4);

	\draw (0,0) node{$\bullet$};
	\draw (0,0) node[below right]{$y$};
	\draw (0,0) -- (1,0) -- (1,1) -- (0,1) -- (0,0);
	\draw (1,1) node{$\bullet$};
	\draw (1,1) node[right]{$u_{\mathbf{C}}$};

	\draw[fill=black] (13, 7) circle (5);
	\draw[fill=black] (9, 1) circle (3);
	\draw[fill=black] (7, 3.3) circle (2);
	\draw[fill=black] (13, 7) circle (5);
	\draw[fill=black] (8, -1.5) circle (2.5);
	\draw[fill=black] (8, 4) circle (2.3);
	\draw[fill=black] (6, 5.5) circle (1.5);
	\draw[fill=black] (5.3, 2) circle (0.4);
	
	\draw (-31.5:7) node{$\mathbf{C}$};
	
	\draw[color=red] (4.9, 1.9) node{$\bullet$};
	\draw[color=red] (4.9, 1.9) node[right]{$v_{\mathbf{C}}$};
	\draw[color=red] (4.9, 1.0) circle (2.3);
	\draw[color=red] (-23 : 4.8) node{$B_{r^*}^{v_\mathbf{C}}$};

\end{tikzpicture}
		\caption{$\mathcal{G}_y$}
	\end{subfigure}%
	\begin{subfigure}[b]{0.6\linewidth}
		\centering
		\begin{tikzpicture}[scale=0.5]

	\clip (-4.5, -4.5) rectangle (12, 7);
	
	\draw[color=cyan!40, fill=cyan!40] (-0.4, -0.6) circle (1);
	\draw[color=cyan!40, fill=cyan!40] (-1.4, 0.2) circle (0.5);
	\draw[color=cyan!40, fill=cyan!40] (-2.8, 1.6) circle (0.8);
	\draw[color=cyan!40, fill=cyan!40] (-2.3, 0.6) circle (0.6);
	\draw[color=cyan!40, fill=cyan!40] (-3.3, 0.4) circle (1);

	\draw (20,0) circle (24);
	\draw (-4.1, 5) node{$B_r^0$};
	
	\draw (0,0) node{$\bullet$};
	\draw (0,0) node[below right]{$y$};
	\draw (0,0) -- (1,0) -- (1,1) -- (0,1) -- (0,0);
	\draw (1,1) node{$\bullet$};
	\draw (1,1) node[right]{$u_{\mathbf{C}}$};

	\draw (6.5,-4.3) circle (8);
	\draw (-116:2.9) node[left]{$B_n^x$};
	
	\draw[fill=black] (13, 7) circle (5);
	\draw[fill=black] (9, 1) circle (3);
	\draw[fill=black] (7, 3.3) circle (2);
	\draw[fill=black] (13, 7) circle (5);
	\draw[fill=black] (8, -1.5) circle (2.5);
	\draw[fill=black] (8, 4) circle (2.3);
	\draw[fill=black] (6, 5.5) circle (1.5);
	\draw[fill=black] (5.3, 2) circle (0.4);
	
	\draw (-31.5:7) node{$\mathbf{C}$};

\end{tikzpicture}
		\caption{$\mathcal{H}_y$}
	\end{subfigure}
	\caption{Visualisation of the four events}
\end{figure}

Conditioned on $\mathbf{C}$, the process $\eta(\mathbb{R}^d\setminus \mathbf{C})$ can be seen as some independent realization of the PPP on $\mathbb{R}^d\setminus\mathbf{C}$. Using the FKG inequality and the fact that, on $\mathcal{E}_y$, the distance between $u_\mathbf{C}$ and $v_\mathbf{C}$ is larger than $r^*$ and smaller than $2n$ yields
\begin{align*}
\mathbb{P}_\lambda[\mathcal{F}_y\cap\mathcal{G}_y\cap\mathcal{H}_y\;\vert\; \mathbf{C}] &\geq  \mathbb{P}_\lambda[\mathcal{F}_y\;\vert\;\mathbf{C}]\cdot\mathbb{P}_\lambda[\mathcal{G}_y\;\vert\;\mathbf{C}]\cdot\mathbb{P}_\lambda[\mathcal{H}_y\;\vert\;\mathbf{C}]\\
&\geq c_{IT}\dfrac{c_5}{(2n)^{2d-2}}\mathbb{P}_\lambda[\mathcal{H}_y\;\vert\;\mathbf{C}]
\end{align*}
almost surely on $\mathcal{E}_y$. Note that if $\mathcal{E}_y$ and $\mathcal{H}_y$ occur, then $\mathcal{P}_x(n)$ occurs since $r^* > \sqrt{d}$. If in addition the events $\mathcal{F}_y$ and $\mathcal{G}_y$ occur, then $B_{r_*}^{v_\mathbf{C}} \cap \mathbf{C}'$ is nonempty. Since by construction 
\[
\Vert v_\mathbf{C} - x\Vert\leq \Vert v_\mathbf{C} - u_\mathbf{C}\Vert + \Vert u_\mathbf{C} - x\Vert \leq n + 2n = 3n,
\]
we deduce that $d(\mathbf{C}\cap B_{3n}^x,\mathbf{C}') < r^*$. Hence, the above inequality yields
\[
\mathbb{P}_\lambda[\mathcal{P}_x(n)\text{ and } d(\mathbf{C}\cap B_{3n}^x,\mathbf{C}') < r^*] \geq \mathbb{P}_\lambda[\mathcal{E}_y\cap \mathcal{F}_y\cap \mathcal{G}_y\cap \mathcal{H}_y] \geq \dfrac{c_{15}}{n^{2d-2}} \mathbb{P}_\lambda[\mathcal{E}_y\cap\mathcal{H}_y].
\]
To control the right hand side, observe that if $\{\mathcal{P}_x(n)\text{ and } d(\mathbf{C}\cap B_{3n}^x,\mathbf{C}') \geq r^*\}$ occurs, there exists some $y\in \mathbb{Z}^d$ with $S_y\cap B_n^x\neq \emptyset$ such that $\mathcal{E}_y\cap\mathcal{H}_y$ occurs. Summing over all possibilities leads to the union bound
\[
c_{16}n^d \mathbb{P}_\lambda[\mathcal{P}_x(n)\text{ and } d(\mathbf{C}\cap B_{3n}^x,\mathbf{C}') < r^*] \geq \dfrac{c_{15}}{n^{2d-2}} \mathbb{P}_\lambda[\mathcal{P}_x(n)\text{ and } d(\mathbf{C}\cap B_{3n}^x,\mathbf{C}') \geq r^*].
\]
The claim follows from
\[
\mathbb{P}_\lambda[\mathcal{P}_x(n)\text{ and } d(\mathbf{C}\cap B_{3n}^x,\mathbf{C}') \geq r^*] = \mathbb{P}_\lambda[\mathcal{P}_x(n)] - \mathbb{P}_\lambda[\mathcal{P}_x(n)\text{ and }d(\mathbf{C}\cap B_{3n}^x,\mathbf{C}) < r^*].
\]

The most delicate point is the use of the FKG inequality for the conditional probability. To avoid any imprecisions, I will present an approximation argument which uses the overall structure of the proof. 

Let $\epsilon > 0$. We consider the following approximation. We say for $x\in\epsilon\mathbb{Z}^d$ that $x + [0,\epsilon)^d$ is open if and only if $\mathcal{O}(\eta)\cap (x + [0,\epsilon)^d)\neq \emptyset$. In this model, we define the connected component $\mathbf{C}(\epsilon)$ of $0$. Then, define the events
\[
\mathcal{E}_y^\epsilon := \{ \mathbf{C}(\epsilon)\cap B_n^x \neq \emptyset\text{ and } d(u_{\mathbf{C}}, \mathbf{C}(\epsilon)) \geq r^* + 2\epsilon\}
\]
and
\[
\mathcal{F}_y^\epsilon := \mathcal{F}_y = \mathcal{D}_{u_{\mathbf{C}}},\quad \mathcal{G}_y^\epsilon := \{u_{\mathbf{C}} \overset{\mathbb{R}^d\setminus\mathbf{C}(\epsilon)}{\leftrightarrow} B_{r^*}^{v_{\mathbf{C}}}\},\quad \mathcal{H}_y^\epsilon := \{ S_y\cap B_n^x\overset{\mathbb{R}^d\setminus\mathbf{C}(\epsilon)}{\leftrightarrow} \partial B_r\}.
\]
Observe that the events $\mathcal{G}_y^\epsilon$ and $\mathcal{H}_y^\epsilon$ only use the approximation of $\mathbf{C}(\epsilon)$ and do not use any approximation of $\eta$ outside of $\mathbf{C}(\epsilon)$. Indeed, the connection $\leftrightarrow$ is still to be thought with respect to $\eta$. Since we are now on a countable state space, we can condition on the event $\{\mathbf{C}(\epsilon) = C\}$. Conditioned on this event, the process $\eta(\mathbb{R}^d\setminus C)$ can be seen as an independent Poisson point process on $\mathbb{R}^d\setminus C$. Hence, the standard FKG-inequality applies and we can write
\[
\mathbb{P}_\lambda[\mathcal{F}_y^\epsilon\cap\mathcal{G}_y^\epsilon\cap\mathcal{H}_y^\epsilon\;\vert\; \mathbf{C}(\epsilon)] \geq \mathbb{P}_\lambda[\mathcal{F}_y^\epsilon\;\vert\;\mathbf{C}(\epsilon)]\cdot \mathbb{P}_\lambda[\mathcal{G}_y^\epsilon\;\vert\; \mathbf{C}(\epsilon)]\cdot\mathbb{P}_\lambda[\mathcal{H}_y^\epsilon\;\vert\;\mathbf{C}(\epsilon)]\quad\text{ a.s. on $\mathcal{E}_y^\epsilon$.}
\]
From here, we get as before
\begin{align*}
\mathbb{P}_\lambda[\mathcal{E}_y^\epsilon\cap\mathcal{F}_y^\epsilon\cap \mathcal{G}_y^\epsilon\cap \mathcal{H}_y^\epsilon] &\geq c_{IT}\dfrac{c_5}{(2n)^{2d-2}} \mathbb{E}_\lambda\left[ 1_{\mathcal{E}_y^\epsilon}\cdot  \mathbb{P}_\lambda[\mathcal{H}_y^\epsilon\;\vert\; \mathbf{C}(\epsilon)] \right]\\
&= c_{IT}\dfrac{c_5}{(2n)^{2d-2}} \mathbb{E}_\lambda\left[ 1_{\mathcal{E}_y^\epsilon}\cdot  1_{\mathcal{H}_y^\epsilon} \right]\\
&= c_{IT}\dfrac{c_5}{(2n)^{2d-2}} \mathbb{P}_\lambda[\mathcal{E}_y^\epsilon\cap \mathcal{H}_y^\epsilon].
\end{align*}
Thus, it suffices to show that this approximated equation has the right convergence. To this end, note that no two balls in $\eta$ are tangent (see Lemma \ref{lem:tangent_balls_0}). Hence, if $\epsilon > 0$ small enough, in a compact neighbourhood of $0$, two balls are connected if and only if they are in the approximation. Hence, $\mathbf{C}(\epsilon)$ converges to $\mathbf{C}$ in the Hausdorff sense on every compact neighbourhood of the origin. We obtain the limit inequality
\[
\mathbb{P}_\lambda[\mathcal{E}_y\cap\mathcal{F}_y\cap\mathcal{G}_y\cap\mathcal{H}_y] \geq c_{IT}\dfrac{c_5}{(2n)^{2d-2}} \mathbb{P}_\lambda[\mathcal{E}_y \cap \mathcal{H}_y]
\]
and we can conclude as before.

The last step is very similar to the discrete case. We simply adapt the algorithms to the continuous model. For $0\leq s\leq r$, take $i_0 := other$ and reveal $\eta_{other}$. At each step $t$, suppose that that the indices $\{i_0,\dots,i_{t-1}\} \in I_L\cup \{other\}$ have been revealed. Then,
\begin{itemize}
	\item If there exists $(x,n)\in I_L\setminus\{i_0,\dots,i_{t-1}\}$ such that the Euclidean distance of $S_x$ to the connected component of $\partial B_s$ in $\mathcal{O}(\eta_{i_0}\cup\dots\cup\eta_{i_{t-1}})$ is smaller than $n + 1$, then set $i_{t} := (x,n)$. If more than one exists, choose one with respect to some fixed ordering.
	\item If such index does not exist, stop the algorithm.
\end{itemize}
Denote this algorithm by $T_s$. Then
\[
\delta_{(x,n)}(T_s) \leq \mathbb{P}_\lambda[S_x^n\leftrightarrow\partial B_s],
\]
where
\[
S_x^n := \bigcup \{S_y\;\vert\; \exists z\in S_y,\; \Vert x - z\Vert \leq n + 1\}.
\]
Hence,
\[
\theta_r(1-\theta_r) \leq c_4\theta_r'(\lambda)\sum_{n\in\mathbb{N}} \int_0^r \mathbb{P}_\lambda[S_x^n\leftrightarrow\partial B_s]\d{s}\cdot(n+1)^{4d-2}\nu[n,n+1).
\]
Now, write $Y$ for the subset of $\mathbb{Z}^d$ such that $S_x^n = \bigcup_{y\in Y} S_y$. Then
\[
\mathbb{P}_\lambda[S_x^n\leftrightarrow \partial B_s] \leq \sum_{y\in Y} \mathbb{P}_\lambda[S_y\leftrightarrow \partial B_s] \leq \dfrac{1}{c_{IT}}\mathbb{P}_\lambda[y\leftrightarrow \partial B_s],
\]
where we used the FKG inequality in the second inequality. Integrating the probability on the right hand side from $0$ to $r$ yields
\[
\int_0^r \mathbb{P}_\lambda[y\leftrightarrow\partial B_s]\d{s} \leq \int_0^r \theta_{\vert s - \Vert y \Vert\vert}(\lambda) \d{s} \leq 2\Sigma_r(\lambda).
\]
Using the fact that $\vert Y\vert$ is proportional to $(n+1)^d$, we can write
\begin{align*}
r\theta_r(1-\theta_r) &\leq 2c_4\Sigma_r\theta_r'\sum_{n\in\mathbb{N}} (n+1)^{5d - 2}\nu[n,n+1) = c\Sigma_r\theta_r'
\end{align*}
by the moment assumption. We conclude by applying Lemma \ref{lem:diff_ineq_cont}.
\end{proof}


\part*{Appendix}

\begin{appendices}

\section{The Poisson Point Process in the Poisson-Boolean model}

The following theorem formalizes our heuristic approach to the problem.

\begin{theo}\label{aprop:approximation}
Consider an independent family $(B_{x}^\epsilon; 0 < \epsilon < \epsilon_0, x\in\epsilon\mathbb{Z}^d)$ such that $B_{x}^\epsilon$ has a Bernoulli law with parameter $\lambda\epsilon^d$. Then 
\[
\lim_{\epsilon\downarrow 0} \sum_{x\in\epsilon\mathbb{Z}^d} \delta_x1_{B_{x}^\epsilon = 1} = \eta\quad\text{ in law},
\]
where $\eta$ is a PPP on $\mathbb{R}^d$ with intensity $\lambda\cdot\mathrm{Leb}$.
\end{theo}
\begin{proof}
Denote the left hand sum by $\eta_\epsilon$ and consider some bounded Borel set $A\subseteq\mathbb{R}^d$ and $k\in\mathbb{N}$. Denote by $\mathcal{F}_\epsilon$ the $\sigma$-algebra generated by the family $(B_x^\epsilon)_{x\in\epsilon\mathbb{Z}^d}$. It is 
\[
\{\eta_\epsilon(A) = k\} = \bigcup_{I\in \binom{A_\epsilon}{k}} \left(\bigcap_{x\in I} \{B_x^\epsilon = 1\}\cap \bigcap_{x\in A_\epsilon\setminus I} \{B_x^\epsilon = 0\}\right) \in \mathcal{F}_\epsilon,
\]
where $A_\epsilon := A\cap\epsilon\mathbb{Z}^d$. Hence, $\eta_\epsilon$ is a point process. Furthermore, being a sum of Bernoulli random variables, one has
\begin{align*}
\eta_\epsilon (A) \sim \mathrm{Bin}(\vert A_\epsilon\vert, \lambda\epsilon^d).
\end{align*}
As we discussed in the introduction, $\vert A_\epsilon\vert \epsilon^d \to \mathrm{Leb}(A)$. Hence $\eta_\epsilon(A)$ converges in law to a Poisson distributed variable with parameter $\lambda\cdot\mathrm{Leb}(A)$. Furthermore, if we take disjoint bounded Borel sets $A_1,\dots,A_n$, then $\eta_\epsilon(A_1),\dots,\eta_\epsilon(A_n)$ are independent by definition. We conclude that $\eta(A_1),\dots,\eta(A_n)$ are independent too. Thus, $\eta$ is a Poisson point process on $\mathbb{R}^d$ with intensity $\lambda\cdot\mathrm{Leb}$.
\end{proof}

Even if this theorem justifies our initial approach, it is not the most useful approximation theorem. We will now see how we can use approximation in percolation theory to prove the measurability of $\{x\leftrightarrow y\}$. But before we get there, we have to do some preliminary work. By $\eta$, we denote a PPP on $\mathbb{R}^d\times\mathbb{R}_+$ with itensity $\lambda\mathrm{d}z\otimes \nu$.

\begin{lem}
Let $R > 0$ and denote by $N$ the number of balls of $\eta$ intersecting $B_R^0$. Then $N$ is a random variable. Furthermore, it is a.s.~finite if and only if $\nu$ has a finite $d$-th moment.
\end{lem}
\begin{proof}
Let $k\in\mathbb{N}$. Define the Borel set $A = \{ (x,r)\in\mathbb{R}^d\times\mathbb{R}_+\;\vert\; \Vert x\Vert \leq R+r\}$. Then
\[
\{N = k\} = \{ \eta(A) = k\}.
\]
Hence, $N$ is a random variable. Moreover, $N$ is finite a.s.~if and only if $\lambda\mathrm{d}z\otimes \nu (A) < +\infty$. We conclude with
\[
\lambda\mathrm{d}z\otimes\nu(A) = \int_0^{+\infty} \lambda\mathrm{Leb}\left(B_{R+r}^0\right)\d{\nu(r)} = \lambda v_d\int_0^{+\infty} (R+r)^d\d{\nu(r)},
\]
where $v_d$ is the volume of the unit ball which depends only on the dimension $d$.
\end{proof}

\begin{prop}\label{aprop:pb_non_trivial}
The two following assertions are equivalent.
\begin{enumerate}[label=\roman*)]
	\item The law $\nu$ has a finite $d$-th moment.
	\item The Poisson-Boolean model is non trivial, i.e.~$\mathcal{O}(\eta)\neq \mathbb{R}^d$ with positive probability.
\end{enumerate}
Moreover, in this case $\mathcal{O}(\eta)\neq\mathbb{R}^d$ almost surely.
\end{prop}
\begin{proof}
Similarly to the above, $\{ B_1^0\subseteq\mathcal{O}(\eta)\} = \{\eta(A) \geq 1\}$, where
\[
A = \{(x,r)\in\mathbb{R}^d\times\mathbb{R}_+\;\vert\; \Vert x\Vert \leq r - 1\},
\]
hence measurable and
\[
\mathbb{P}[B_1^0\subseteq\mathcal{O}(\eta)] = \exp\left(- \lambda v_d \int_1^{+\infty} (r-1)^d\d{\nu(r)}\right).
\]
This quantity is 1 if and only if $\nu$ has no $d$-th moment. In the same way, $\{\mathcal{O}(\eta) = \mathbb{R}^d\}$ is measurable. Furthermore, this event is obviously translation invariant. By ergodicity, we conclude that its probability is either 0 or 1. And since it is included in $\{B_1^0\subseteq\mathcal{O}(\eta)\}$, the statement follows.
\end{proof}

\textbf{From now on, we will only consider the case that $\nu$ has a finite $d$-th moment.} Now, we will use approximation arguments to show the measurability of elementary events. To this end, we define for $\epsilon > 0$ the approximation $\eta_\epsilon\subseteq\mathbb{R}^d\times\mathbb{R}_+$ as the point process of points $x\in\epsilon(\mathbb{Z}^d\times\mathbb{N})$ such that $\eta( R_x^\epsilon) \geq 1$, where $R_x^\epsilon = x + [0,\epsilon)^{d+1}$ are the $\epsilon$-boxes from above.

Also, we will define for a point process $\xi$ on $\mathbb{R}^d\times\mathbb{R}_+$ the point process $$\xi^n := \xi\cap \left(B_n^0\times[0,n]\right).$$

\begin{lem}\label{lem:tangent_balls_0}
The set $T:=\{\text{two balls of $\eta$ are tangent}\}$ is measurable of probability 0.
\end{lem}
\begin{proof}Note that
\[
T = \bigcup_{n\geq 1} \underbrace{\{\text{two balls of $\eta^n$ are tangent}\}}_{=:T_n},
\]
hence it suffices to show the measurability of $T_n$. Suppose that the balls $B_r^x$ and $B_s^y$ are tangent. Denote by $x_\epsilon, r_\epsilon,y_\epsilon,s_\epsilon$ their approximations in $\eta_\epsilon^{n+2\epsilon}$. Then $\Vert x-x_\epsilon\Vert < \epsilon$ and similarly for the three other quantities. Thus, the induced balls $B_{r_\epsilon}^{x_\epsilon}$ and $B_{s_\epsilon}^{y_\epsilon}$ are \emph{nearly} tangent with an error of at most $4\epsilon$. If one would like to write this formally, one must distinguish the two cases that one centre is included in the other ball or not. But note that the event that the approximated balls are $4\epsilon$-nearly tangent is $\eta_\epsilon^{n+2\epsilon}$-measurable, because this point process can take at most a finite number of states. Now, if $B_r^x$ and $B_s^y$ are not tangent, then for $\epsilon > 0$ small enough, $B_{r_\epsilon}^{x_\epsilon}$ and $B_{s_\epsilon}^{y_\epsilon}$ are not $4\epsilon$-nearly tangent anymore. Hence,
\[
T = \bigcap_{n\geq 1} \llbrace\text{two balls of $\eta_{2^{-n}}^{n+2^{1-n}}$ are $2^{2-n}$-nearly tangent}\rrbrace.
\]
Hence $T$ is measurable. Applying Mecke's formula (cf. Theorem \ref{theo:mecke}) twice yields
\begin{align*}
\mathbb{P}[T_n] &= \mathbb{P}\left[\exists \{(x_1,r_1),(x_2,r_2)\}\in\binom{\eta^n}{2},\, (x_1, r_1)\text{ and }(x_2,r_2)\text{ are tangent}\right]\\
&\leq \mathbb{E}\left[\sum_{(x_1,r_1)\in\eta^n}\sum_{(x_2,r_2)\in\eta^n} 1_{(x_1,r_1)\neq (x_2, r_2)}\cdot 1_{(x_1,r_1)\text{ and }(x_2,r_2)\text{ are tangent}} \right]\\
&= \lambda\int_{B_n^0}\int_0^n \mathbb{E}\left[\sum_{(x_2,r_2)\in\eta^n \cup\{(x_1,r_1)\}} 1_{(x_1,r_1)\neq (x_2, r_2)}\cdot 1_{(x_1,r_1)\text{ and }(x_2,r_2)\text{ are tangent}} \right]\d{\nu(r_1)}\d{x_1}\\
&= \lambda\int_{B_n^0}\int_0^n \mathbb{E}\left[\sum_{(x_2,r_2)\in\eta^n} 1_{(x_1,r_1)\neq (x_2, r_2)}\cdot 1_{(x_1,r_1)\text{ and }(x_2,r_2)\text{ are tangent}} \right]\d{\nu(r_1)}\d{x_1}\\
&= \lambda^2\int_{B_n^0}\int_0^n \int_{B_n^0}\int_0^n \mathbb{P}[(x_1,r_1)\text{ and }(x_2,r_2)\text{ are neither equal nor tangent}]\d{\nu(r_2)}\d{x_2}\d{\nu(r_1)}\d{x_1}\\
&= \lambda^2\int_{B_n^0}\int_0^n \int_{B_n^0}\int_0^n 1_{(x_1,r_1)\text{ and }(x_2,r_2)\text{ are neither equal nor tangent}}\d{\nu(r_2)}\d{x_2}\d{\nu(r_1)}\d{x_1}\\
&= 0.
\end{align*}
Hence,
\[
\mathbb{P}[T] = \lim_n \mathbb{P}[T_n] = 0.
\]
\end{proof}

We will now show that the measurability of the most basic event.

\begin{prop}
Let $x,y\in\mathbb{R}^d$. Then $\{x\leftrightarrow y\}$ is measurable.
\end{prop}
\begin{proof}
Throughout the proof, we will assume that no two balls in $\eta$ are tangent. Let $x,y\in\mathbb{R}^d$ be distinct. (Otherwise the proof is trivial.) Since
\[
\{x\leftrightarrow y\} = \bigcup_{n\geq 1}\bigcup_{m\geq 1}\{ x\overset{\eta^n}{\leftrightarrow} y\}\cap\{\vert \eta^n\vert = m\}
\]
it suffices to show that $\{x\overset{\eta^n}{\leftrightarrow}y\}\cap\{\vert\eta^n\vert = m\}$ is measurable for every $n\geq 1$ and $m\geq 1$. So fix $n\geq 1$ and $m\geq 1$. Since we consider only a finite number of balls, we can assume for $\epsilon > 0$ small enough that two balls intersect if and only if there $\epsilon$-approximations intersect. As before, the approximated event is trivially measurable. We conclude the proof by noting that
\[
\{x\overset{\eta^n}{\leftrightarrow}y\}\cap\{\vert\eta^n\vert = m\} = \liminf_{k} \left(\left\{x\overset{\eta^{n+2k}_{1/k}}{\leftrightarrow} y\right\}\cap\{\vert\eta^n\vert = m\}\right).
\]
\end{proof}

\end{appendices}

\nocite{*}
\phantomsection
\printbibliography[title={References}, heading=bibintoc]
\phantomsection

\end{document}